\documentclass[11pt]{amsart}
\usepackage{amsthm,amssymb,mathrsfs,mathtools,empheq}
\usepackage[shortlabels]{enumitem}
\usepackage[T1]{fontenc}
\usepackage{bm}
\usepackage[notref,notcite,final]{showkeys}
\mathtoolsset{showonlyrefs}
\usepackage{fullpage,color}

\newtheorem{theorem}{Theorem}[section]
\newtheorem*{theorem*}{Theorem}

\newtheorem{lemma}[theorem]{Lemma}
\newtheorem{proposition}[theorem]{Proposition}
\newtheorem*{proposition*}{Proposition}

\newtheorem*{conjecture*}{Conjecture}

\theoremstyle{definition}
\newtheorem{definition}[theorem]{Definition}
\newtheorem{remark}[theorem]{Remark}

\numberwithin{equation}{section}

\def\bN {\mathbb{N}}

\def\bR {\mathbb{R}}

\def\bZ {\mathbb{Z}}

\def\cE {\mathcal{E}}

\def\cH {\mathcal{H}}

\def\cZ {\mathcal{Z}}

\def\scrL{\mathscr{L}}

\def\la {\langle}
\def\ra {\rangle}

\newcommand{\tx}[1]{\mathrm{#1}}

\newcommand{\wt}[1]{\widetilde{#1}}
\newcommand{\bs}[1]{\boldsymbol{#1}}
\newcommand{\conj}[1]{\overline{#1}}
\newcommand{\sh}[1]{#1^\sharp}

\renewcommand{\ker}{\operatorname{ker}}

\newcommand{\eee}{\mathrm e}

\newcommand{\ssubset}{\subset\joinrel\subset}

\newcommand{\ud}{\mathrm{\,d}}
\newcommand{\vd}{\mathrm{d}}

\newcommand{\vD}{\mathrm{D}}
\newcommand{\dd}[1]{{\frac{\vd}{\vd{#1}}}}

\title{Kink networks for scalar fields in dimension $1+1$}
\author{Gong Chen}
\address{Department of Mathematics, University of Toronto, 40 St George St, Toronto, Ontario, Canada}
\email{gc@math.toronto.edu}
\author{Jacek Jendrej}
\address{CNRS and Universit\'e Sorbonne Paris Nord, LAGA, UMR 7539, 99 av J.-B.~Cl\'ement, 93430 Villetaneuse, France }
\email{jendrej@math.univ-paris13.fr}

\begin{document}

\begin{abstract}
We consider a scalar field equation in dimension $1+1$ with a positive external potential having
non-degenerate isolated zeros.
We construct weakly interacting pure multi-solitons, that is solutions converging exponentially in time to a superposition
of Lorentz-transformed kinks, in the case of distinct velocities.
We find that these solutions form a $2K$-dimensional smooth manifold in the space of solutions,
where $K$ is the number of the kinks.
We prove that this manifold is invariant under the transformations corresponding to the invariances of the equation,
that is space-time translations and Lorentz boosts.
\end{abstract}

\maketitle
\section{Introduction}
\label{sec:intro}
\subsection{Setting of the problem}
\label{ssec:setting}
We study scalar fields in dimension $1+1$.
Let $W: \bR \to [0, +\infty)$ be a function of class $C^\infty$ and consider the Lagrangian action
\begin{equation}
\label{eq:lagrange}
\scrL(\phi) := \iint \Big(\frac 12(\partial_t \phi)^2 - \frac 12 (\partial_x\phi)^2 - W(\phi)\Big)\, \ud  x \ud t,
\end{equation}  
for real-valued functions $\phi = \phi(t, x)$. The Euler-Lagrange equation associated with $\scrL$ is the nonlinear wave equation 
\begin{equation}
\label{eq:csf}
\partial_t^2 \phi(t, x) - \partial_x^2 \phi(t, x) + W'(\phi(t, x)) = 0, \qquad (t, x) \in \bR\times \bR,\ \phi(t, x) \in \bR.
\end{equation}
We study~\eqref{eq:csf} for potentials $W$ satisfying the following conditions:
\begin{enumerate}
\item[(A1)] $W(\phi) \geq 0$ for all $\phi \in \bR$,
\item[(A2)] $\Omega := W^{-1}(0) \subset \bR$ has no accumulation points,
\item[(A3)] $W''(\omega) > 0$ for all $\omega \in \Omega$.
\item[(A4)] $\int_0^\infty \sqrt{W(\phi)}\ud\phi = \int_{-\infty}^0\sqrt{W(\phi)}\ud\phi = \infty$.
\end{enumerate}
Any $\omega \in \Omega$ is called a \emph{vacuum}. We sort them in the increasing order,
so that $\Omega = \{\omega_n\}_{n \in I}$, $I \subset \bZ$, $\omega_n < \omega_{n+1}$
for all $n, n+1 \in I$.
Typical examples include the $\phi^4$ model $W(\phi) = (1 - \phi^2)^2$,
the $\phi^6$ model $W(\phi) = \phi^2(1 - \phi^2)^2$, and
the sine-Gordon model $W(\phi) = 1 - \cos(\phi)$.

We denote $(\phi_0, \dot \phi_0)$ a generic element of the phase space.
The potential energy $E_p$, the kinetic energy $E_k$ and the total energy $E$ are given by
\begin{align}
E_p(\phi_0) &= \int_{\bR}\Big(\frac 12 (\partial_x\phi_0)^2 + W(\phi_0)\Big)\ud x, \\
E_k(\dot \phi_0) &= \int_{\bR}\frac 12( \dot\phi_0)^2 \ud x, \\
E(\phi_0, \dot\phi_0) &= \int_{\bR}\Big(\frac 12(\dot \phi_0)^2+\frac 12 (\partial_x\phi_0)^2 + W(\phi_0)\Big)\ud x.
\end{align}
The set of finite energy states $(\phi_0, \dot \phi_0)$ is a union of $|I|^2$ affine spaces, called \emph{sectors},
\begin{equation}
\label{eq:sectors}
\cE_{m, n} := \{(\phi_0, \dot \phi_0): E(\phi_0, \dot \phi_0) < \infty\ \text{and}\ \lim_{x \to -\infty}\phi_0(x) = \omega_m, \lim_{x \to \infty}\phi_0(x) = \omega_n\},
\end{equation}
each of which is parallel to the linear space $\cE := H^1(\bR) \times L^2(\bR)$ which we call the \emph{energy space},
see Section~\ref{ssec:sectors}.

Equation \eqref{eq:csf} admits static solutions. They are the critical points
of the potential energy. The trivial ones are the vacuum fields $\phi(t, x) = \omega_n$ for some $n \in I$.
The solution $\phi(t, x) = \omega_n$ has zero energy and is the ground state in $\cE_{n, n}$.

There are also non-constant static solutions $\phi(t, x)$
connecting two consecutive vacua, that is
\begin{equation}
\label{eq:static-nn'}
\lim_{x \to -\infty}\phi(t, x) = \omega_{n}, \quad \lim_{x \to \infty}\phi(t, x) = \omega_{n'},\quad\text{for some }n, n' \in I\text{ and }|n - n'| = 1.
\end{equation}
One can discribe all these solutions.
For all $n, n' \in I$ with $|n - n'| = 1$ there exists a function $H_{n, n'}(x)$
such that all the static solutions satisfying \eqref{eq:static-nn'} are $\phi(t, x) = H_{n, n'}(x - a)$ for some $a \in \bR$.
These solutions are the ground states in $\cE_{n, n'}$:
\begin{equation}
\label{eq:ground-kink}
\inf_{(\phi_0, \dot \phi_0) \in \cE_{n, n'}}E(\phi_0, \dot \phi_0) = E_p(H_{n,n'}).
\end{equation}
If $n' = n+1$, we call these static solutions the \emph{kinks}. If $n' = n-1$, we call them \emph{antikinks}.
The functions $H_{n, n'}$ are studied in detail in Section~\ref{ssec:static}.

It is not difficult to see that for $n' - n \geq 2$ we have
\begin{equation}
\inf_{(\phi_0, \dot \phi_0) \in \cE_{n, n'}}E(\phi_0, \dot \phi_0) = \inf_{(\phi_0, \dot\phi_0) \in \cE_{n', n}}E(\phi_0, \dot\phi_0) = \sum_{l=n}^{n'-1}E_p(H_{l,l+1}),
\end{equation}
but the infimum is not attained, and there is no ground state in $\cE_{n, n'}$ or $\cE_{n', n}$.
Constant solutions, kinks and antikinks are all the stationary states of \eqref{eq:csf}, see Section~\ref{ssec:static}.

\begin{remark}
There is in general no ``canonical'' choice of $H_{n, n'}$
among the family of its space translates.
For every pair $n, n' \in I$ such that $|n - n'| = 1$, we make an arbitrary choice.
\end{remark}

An important property of the equation \eqref{eq:csf} is the invariance by space-time translations and Lorentz transformations
\begin{equation}
\label{eq:lorentz-def}
\begin{gathered}
(t, x) = \big(t_0 + \gamma(t' + vx'), x_0 + \gamma(x' + vt')\big)\quad\Leftrightarrow\\
\Leftrightarrow \quad(t', x') = \big(\gamma(t - t_0 - v(x - x_0)), \gamma(x - x_0 - v(t - t_0))\big),
\end{gathered}
\end{equation}
where $(t_0, x_0) \in \bR^2$, $-1 < v < 1$ and $\gamma := (1-v^2)^{-1/2}$,
that is, if $\psi(t', x')$ is a solution of \eqref{eq:csf} in some region of the $(t', x')$-space-time,
then $\phi(t, x) := \psi(t', x')$ is a solution of \eqref{eq:csf} as well,
in the corresponding region of the $(t, x)$-space-times.
Applying this transformation to the static solutions found above,
we obtain the moving kinks and antikinks
\begin{equation}
\phi(t, x) = H_{n, n'}(\gamma(x - vt-a)), \qquad |n - n'| \leq 1.
\end{equation}
One can easily check that these are all the travelling waves, that is all the finite-energy solutions of \eqref{eq:csf}
such that $\phi(t, x) = \psi(x - vt)$ for some function $\psi: \bR \to \bR$ and some $v \in \bR$.
\begin{definition}
Let $K \in \{0, 1, 2, \ldots\}$. If $\bs n = (n_0, n_1, \ldots, n_K) \in I^{K+1}$
is a sequence such that $|n_{k-1} - n_k| = 1$ for all $k \in \{1, \ldots, K\}$,
we say that $\bs n$ is a \emph{chain of vacua}.
We say that
\begin{equation}
S^{(K)} := \{\bs v = (v_1, \ldots, v_K) \in \bR^K: -1 < v_1 < \ldots < v_K < 1\}
\end{equation}
is the set of \emph{admissible velocities}.
\end{definition}
Fix $K\in \{0, 1, 2, \ldots\}$ and a chain of vacua $\bs n$.
Given $\bs v = (v_1, \ldots, v_K) \in S^{(K)}$ and $\bs a = (a_1, \ldots, a_K)\in\bR^K$, we set
\begin{equation}
\label{def profile}
H(\bs v, \bs a; t, x) := \omega_{n_0} + \sum_{k=1}^K \big(H_{n_{k-1}, n_k}(\gamma_k(x - v_k t - a_k)) - \omega_{n_{k-1}}\big),
\end{equation}
where $\gamma_k := (1 - v_k^2)^{-\frac 12}$ is the Lorentz factor.
Thus, for $t$ large, $H(\bs v, \bs a)$
is a superposition of translated and Lorentz-transformed kinks, separated by large distances.
\subsection{Statement of the results and comments}
\label{ssec:results}
Our first goal is to construct smooth multi-soliton solutions,
that is solutions converging to $H(\bs v, \bs a; t, x)$ in the energy space.
The function space relevant for us is the space of functions defined for large times
whose energy decays exponentially. For $T \in \bR$ and $\delta > 0$, we set
\begin{equation}
\|\psi\|_{\cH_{T, \delta}}^2 := \sup_{t > T}\ \eee^{2\delta t}\int_\bR\big(|\partial_t \psi(t, x)|^2
+ |\partial_x \psi(t, x)|^2 + |\psi(t, x)|^2\big)\ud x < \infty.
\end{equation}
We can state our result on the existence and uniqueness of solutions which converge to multi-kinks exponentially in time.
\begin{theorem}
\label{thm:construct}
For all $K \in \bN_0$, chain of vacua $\bs n$, $\bs v \in S^{(K)}$ and $\bs a \in \bR^K$,
there exist $T_0 \in \bR$, $\delta_0 > 0$ and $\Psi(\bs v, \bs a; \cdot, \cdot) \in \cH_{T_0, \delta_0}$
such that
\begin{equation}
\phi(t, x) := H(\bs v, \bs a; t, x) + \Psi(\bs v, \bs a; t, x)
\end{equation}
is a solution of \eqref{eq:csf}.

If $T \in \bR$, $\delta > 0$ and $\psi \in \cH_{T, \delta}$ is such that $\phi(t, x) = H(\bs v, \bs a; t, x) + \psi(t, x)$
is a solution of \eqref{eq:csf}, then $\psi(t, x) = \Psi(\bs v, \bs a; t, x)$.
\end{theorem}
Next, we consider the action of the transformations \eqref{eq:lorentz-def} on the set of the constructed solutions.
This is not a trivial task, since there is no absolute time,
hence functions decaying exponentially in time $t$ could have some other behaviour with respect to another time $t'$.
It is easy to check that, under this transform, $H(\bs v, \bs a)$ is mapped to $H(\bs v', \bs a')$, where
\begin{equation}
\label{eq:prime-formula}
v_j' := \frac{v_j - v}{1 - v_jv}, \qquad a_j' := (\gamma_j')^{-1}\gamma_j(a_j + v_j t_0 - x_0)
\end{equation}
(the formula for $v_j'$ is the well-known velocity-addition formula in Special Relativity).
\begin{theorem}
\label{thm:lorentz}
The set of the solutions constructed in Theorem~\ref{thm:construct} is invariant under the transformations \eqref{eq:lorentz-def}:
if $(t, x)$ and $(t', x')$ are related by \eqref{eq:lorentz-def}, then
\begin{equation}
\label{eq:lorentz-inv}
H(\bs v, \bs a; t, x) + \Psi(\bs v, \bs a; t, x) = H(\bs v', \bs a'; t', x') + \Psi(\bs v', \bs a'; t', x'),
\end{equation}
where $\bs v'$ and $\bs a'$ are given by \eqref{eq:prime-formula}.
\end{theorem}

Finally, we address the question of smoothness of these objects.
\begin{theorem}
\label{thm:smooth}
If $K \in \bN_0$, $\bs n$ is a chain of vacua and $A \ssubset S^{(K)} \times \bR^K$ is an open set with compact closure,
then $(T_0, \delta_0)$ in Theorem~\ref{thm:construct} can be chosen uniformly for all $(\bs v, \bs a) \in A$.
The function $\Psi: A \times (T_0, \infty) \times \bR \to \bR$ constructed in Theorem~\ref{thm:construct}
is of class $C^\infty$ and all its partial derivatives decay exponentially in $t \to \infty$.
\end{theorem}
\begin{remark}
Uniqueness of pure multi-kink solutions could also be proved, modulo minor technicalities,
under the assumption of convergence faster than any power of $t$ (instead of the exponential convergence).
Uniqueness under \emph{no} assumptions on the rate of convergence remains an interesting open problem.
Note that in the elliptic case this was achieved by del~Pino, Kowalczyk and Pacard \cite{Pino-moduli},
and for KdV and generalised KdV by Martel \cite{Martel05}.
\end{remark}
\begin{remark}
We can say that $S^{(K)} \times \bR^K$ is the \emph{parameter manifold} of multikinks
in the forward time direction for a given chain of vacua (under the assumption of exponential in time convergence),
and $\Phi$ defines the \emph{solution map}.
The Lorentz transforms act on this parameter manifold,
and the quotient manifold parametrises the multikink solutions,
up to invariances of the equation.
It has dimension $0$ for $K = 1$ and dimension $2K - 3$ if $K \geq 2$.
\end{remark}
%
\begin{remark}
In the statement of Theorem~\ref{thm:lorentz}, $\Psi(\bs v, \bs a; t, x)$ and $\Psi(\bs v', \bs a'; t', x')$
are, a priori, well-defined only if both $t$ and $t'$ are large enough,
but in fact, under our assumptions, the problem is globally well-posed,
so that the Lorentz transform is well-defined.
\end{remark}
\begin{remark}
A well-known open problem is to describe the \emph{collisions},
that is the behaviour of a~given
multikink as $t \to -\infty$, see for example \cite{Goodman} for a mathematically non-rigorous discussion.
It is expected that, for generic $W$,
the collisions are inelastic, meaning that for $t \to -\infty$
the solution has a different behaviour than for $t \to \infty$.
In this paper, we only consider the evolution in one time direction.
\end{remark}
\subsection{Comparison with previous works}
Several multi-soliton constructions followed the paper by Martel \cite{Martel05} cited above,
see \cite{Combet, MRT-water}, and in particular \cite{CoMu2014, CoMa-multi}
devoted to the Klein-Gordon equations, very similar to \eqref{eq:csf}.
These works focused on the questions of existence, smoothness and uniqueness of multi-solitons
for given velocity and shift parameters.

Our contribution is to study the solution map, which involves the examination of smooth dependence of solutions
with respect to the velocities and shifts.
The action of Lorentz transforms on multi-solitons for wave equations was not clarified in the existing literature either.

All of these works followed the scheme introduced in \cite{Martel05} and \cite{Merle90},
based on energy estimates, weak compactness and weak continuity of the flow.
Here, we propose an alternative approach, based on the Contraction Principle,
which we believe is a convenient way of proving smooth dependence with respect to velocities and shifts.
Most likely, the weak compactness approach of \cite{Martel05, Merle90} could also be adopted.
We point out, however, that proving smooth dependence on the parameters is a more difficult task
than proving smoothness of each multi-soliton with respect to $t$ and $x$.
Indeed, the latter can be easily deduced from the regularising effect of the equation.
In contrast, we believe that smoothness with respect to the velocities and shifts cannot be directly inferred
from the arguments of \cite{CoMu2014, CoMa-multi}.

\subsection{Acknowledgements}
J.Jendrej was supported by  ANR-18-CE40-0028 project ESSED.

We would like to thank Micha{\l} Kowalczyk for enlightening discussions at an early stage of this project and for his detailed explanations about the paper \cite{Pino-moduli}.

We also thank the referees for suggesting many improvements in the submitted version of the~manuscript.

\section{Kinks, multikinks and their coercivity properties}
\label{sec:static}
\subsection{Finite energy sectors}
\label{ssec:sectors}
We prove that any finite energy state lies in one of the sectors defined by \eqref{eq:sectors}.
\begin{proposition}
\label{prop:sectors}
If $\phi_0 : \bR \to \bR$ is a measurable function which satisfies $E_p(\phi_0) < \infty$,
then $\phi_0 \in C(\bR)$ and there exist $m, n \in I$ such that
$\lim_{x\to -\infty}\phi_0(x) = \omega_m$, $\lim_{x\to \infty}\phi_0(x) = \omega_n$.
\end{proposition}
\begin{proof}
The proof is based on the so-called \emph{Bogomolny trick}.
Consider the auxiliary function
\begin{equation}
\label{eq:Gamma-def}
\Gamma(\phi) := \int_0^\phi \sqrt{2W(y)}\ud y.
\end{equation}
Then $\Gamma$ is a strictly increasing $C^1$ function and
assumption (A4) yields $\lim_{\phi \to \pm \infty}\Gamma(\phi) = \pm\infty$.

Since $\partial_x \phi_0 \in L^2(\bR)$, it is clear that $\phi_0 \in C(\bR)$.
Fix $x_1, x_2 \in \bR$, $x_1 < x_2$. We will check that
\begin{equation}
\label{eq:Gamma-ineq}
\begin{aligned}
|\Gamma(\phi_0(x_2)) - \Gamma(\phi_0(x_1))| &= \int_{x_1}^{x_2}\Big(\frac 12 (\partial_x\phi_0)^2 +
W(\phi_0)\Big)\ud x - \frac 12 \int_{x_1}^{x_2}\big(|\partial_x\phi_0| - \sqrt{2W(\phi_0)}\big)^2\ud x \\
& \leq \int_{x_1}^{x_2}\Big(\frac 12 (\partial_x\phi_0)^2 +
W(\phi_0)\Big)\ud x.
\end{aligned}
\end{equation}
By approximation, we can assume $\phi_0 \in C^1([x_1, x_2])$. The chain rule yields
\begin{equation}
|\partial_x \Gamma(\phi_0(x))| = |\partial_x \phi_0(x)|\sqrt{2W(\phi_0(x))}
= \frac 12 (\partial_x \phi_0(x))^2 + W(\phi_0(x)) - \frac 12\big(|\partial_x\phi_0| - \sqrt{2W(\phi_0)}\big)^2,
\end{equation}
and it suffices to integrate in $x$.

Since $E_p(\phi_0) < \infty$, we deduce that $\lim_{x \to \infty}\Gamma(\phi_0(x))$
exists and is finite. This implies, in turn, that $\lim_{x \to\infty} \phi_0(x)$
exists and is finite. It is immediate from the definition of $E_p$ that
$E_p(\phi_0) < \infty$ implies $\lim_{x \to\infty} \phi_0(x) \in \Omega$.
The situation is analogous for $x \to -\infty$.
\end{proof}
We skip the discussion of the Cauchy theory for \eqref{eq:csf}.
By well-established methods, one can obtain that \eqref{eq:csf} is locally well-posed in each sector $\cE_{m, n}$
and the energy is conserved.
The computation above shows, using again the assumption (A4), that a solution of finite energy is bounded,
which implies the absence of blow-up, hence global well-posedness in each sector. We thus have the following result,
which we state without proof.
\begin{proposition}
\label{prop:cauchy}
For all $m, n \in I$, $(\phi_0, \dot \phi_0) \in \cE_{m ,n}$ and $t_0 \in \bR$ there exists a unique solution
$(\phi, \partial_t \phi) \in C(\bR, \cE_{m ,n})$ of \eqref{eq:csf} such that $(\phi(t_0), \partial_t \phi(t_0)) = (\phi_0, \dot\phi_0)$.
\end{proposition}

\subsection{Stationary solutions}
\label{ssec:static}
For $n, n+1\in I$, we define the \emph{kink} $H_{n, n+1}$ by the formula
\begin{equation}
\label{eq:G-def}
H_{n,n+1}(x) = G_n^{-1}(x), \qquad\text{with}\ \ G_n(\psi) := \int_{\psi_{n,n+1}}^\psi \frac{\ud y}{\sqrt{2W(y)}}\ \ \text{for all }\psi \in (\omega_n, \omega_{n+1}),
\end{equation}
where $\psi_{n,n+1} \in (\omega_n, \omega_{n+1})$ is chosen arbitrarily.
Since, by assumption, $\omega_n$ and $\omega_{n+1}$ are non-degenerate zeros of $W$,
we see that $\lim_{\psi \to \omega_n}G_n(\psi) = -\infty$ and $\lim_{\psi \to \omega_{n+1}}G_n(\psi) = \infty$, thus $H_{n, n+1}$ is a well-defined smooth increasing function, $\lim_{x \to -\infty} H_{n,n+1}(x) = \omega_n$ and $\lim_{x \to \infty} H_{n,n+1}(x) = \omega_{n+1}$.
It satisfies the differential equation
\begin{equation}
\label{eq:kink-eq}
\partial_x H_{n, n+1}(x) = \sqrt{2W(H_{n, n+1}(x))}.
\end{equation}

We also define the \emph{antikink}
\begin{equation}
H_{n+1, n}(x) := H_{n, n+1}(-x),
\end{equation}
which satisfies the differential equation
\begin{equation}
\label{eq:antikink-eq}
\partial_x H_{n+1, n}(x) = -\sqrt{2W(H_{n+1, n}(x))}.
\end{equation}

We denote $m_n := \sqrt{W''(\omega_n)} > 0$ the mass corresponding to the vacuum $\omega_n$.
The exponential decay of $H_{n, n'}(x)$ and its derivatives for $|x|$ large will be essential for our analysis.
\begin{proposition}\label{prop:decaykink}
Let $n, n' \in I$ with $|n - n'| = 1$. The function $H_{n, n'}$ is of class $C^\infty$.
For any $k \in \{0\} \cup \bN$ there exists $C = C(n, n', k) > 0$ such that
\begin{align}
\big|\partial_x^k H_{n, n'}(x) - \delta_{0, k} \omega_n\big| &\leq C\eee^{m_n x}, &\text{for all } x \leq 0, \label{eq:Hnn'-asym}
\\
\big|\partial_x^k H_{n, n'}(x) - \delta_{0, k} \omega_{n'}\big| &\leq C\eee^{-m_{n'} x},&\text{for all } x \geq 0.\label{eq:Hnn'-asym-bis}
\end{align}
where $\delta_{0, 0} = 1$ and $\delta_{0, k} = 0$ for $k > 0$.
\end{proposition}
\begin{proof}
We consider the case $n' = n + 1$ and $x \leq 0$; the remaining cases are reduced
to this one using the symmetries of the problem.

Since $2W(y) = m_n^2 (y - \omega_n)^2 (1 + O(y - \omega_n))$ as $y > \omega_n$ and $y \to \omega_n$, we have $(2W(y))^{-1/2} = (m_n(y - \omega_n))^{-1} + O(1)$,
and \eqref{eq:G-def} yields
\begin{equation}
G_n(\psi) = \frac{1}{m_n}\log(y - \omega_n) + O(1), \qquad y > \omega_n,\ y \to \omega_n.
\end{equation}
Thus,
\begin{equation}
\frac{1}{m_n}\log(H_{n, n+1}(x) - \omega_n) = x + O(1), \qquad x \to -\infty,
\end{equation}
which implies \eqref{eq:Hnn'-asym} for $k = 0$.

The bound for $\partial_x H_{n, n+1}(x)$ follows from \eqref{eq:kink-eq}
and the fact that $\sqrt{2W(\psi)}$ is a locally Lipschitz function.
For $k \geq 2$, \eqref{eq:Hnn'-asym} is proved by induction,
differentiating $k-2$ times $W'(H_{n, n+1}(x))$ using the chain and Leibniz rules.
\end{proof}
\begin{proposition}
\label{prop:static}
All the finite-energy stationary solutions of \eqref{eq:csf} are
\begin{itemize}
\item the vacua $\phi(t, x) = \omega_n$ for some $n \in I$,
\item the kinks $\phi(t, x) = H_{n, n+1}(x - a)$ for some $n, n+1 \in I$ and $a \in \bR$,
\item the antikinks $\phi(t, x) = H_{n+1, n}(x - a)$ for some $n, n+1 \in I$ and $a \in \bR$.
\end{itemize}
\end{proposition}
\begin{proof}
A stationary field $\phi(t, x) = \psi(x)$ is a solution of \eqref{eq:csf} if and only if
\begin{equation}
\label{eq:psi4}
\partial_x^2\psi(x) = W'(\psi(x)),\qquad\text{for all }x\in \bR.
\end{equation}
We seek solutions of \eqref{eq:psi4} such that $E_p(\psi) < \infty$,
in particular $\psi \in C(\bR)$, so \eqref{eq:psi4} and $W \in C^\infty(\bR)$ yield $\psi \in C^\infty(\bR)$.
Multiplying \eqref{eq:psi4} by $\partial_x \psi$ we get
\begin{equation}
\partial_x\Big(\frac 12 (\partial_x \psi)^2 - W(\psi)\Big)
= \partial_x \psi\big(\partial_x^2 \psi - W'(\psi)\big) = 0,
\end{equation}
so $\frac 12 (\partial_x \psi(x))^2 - W(\psi(x)) = k$ is a constant.
But then $E_p(\psi) < \infty$ implies $k = 0$.
We obtain first order autonomous equations, called the Bogomolny equations,
\begin{equation}
\label{eq:bogom}
\partial_x\psi(x) = \sqrt{2W(\psi(x))}\quad\text{or}\quad \partial_x\psi(x) = -\sqrt{2W(\psi(x))},
\quad\text{for all }x \in \bR.
\end{equation}

If there exists $x_0 \in \bR$ such that $\psi(x_0) = \omega \in \Omega$,
then $\psi(x) = \omega$ for all $x \in \bR$. We thus assume $\psi(x) \notin \Omega$ for all $x \in \bR$.

We now argue that the sign in \eqref{eq:bogom} is always the same.
Indeed, the smoothness of $\psi$ implies that the sets $\{x: \partial_x\psi(x) = \sqrt{2W(\psi(x))}\}$
and $\{x: \partial_x\psi(x) = -\sqrt{2W(\psi(x))}\}$ are closed,
and $\psi(x) \notin \Omega$ implies that they are disjoint. Thus one of them is empty.

We consider the case $\partial_x \psi(x) = \sqrt{2W(\psi(x))}$ for all $x \in \bR$,
the other case being analogous. In particular, $\psi$ is an increasing function.

If $\psi(0) = \psi_0 \in (\omega_n, \omega_{n+1})$ for some
$n, n+1 \in I$, then there exists $a \in \bR$ such that $H_{n, n+1}(a) = \psi_0$,
and the uniqueness of solutions of ODEs yields $\psi(x) = H_{n, n+1}(x - a)$ for all $x \in \bR$,
thus $\psi$ is a kink.

If $\psi(0) < \min\Omega$, then we could not have $\lim_{x \to -\infty}\psi(x) \in \Omega$,
since $\psi$ is increasing. Similarly, $\psi(0) > \max \Omega$ is impossible.
\end{proof}

\subsection{Coercivity near a static kink}
\label{ssec:coer-kink}
Let $n, n+1 \in I$ and let $\phi_0: \bR \to \bR$ be a state such that
$E_p(\phi_0) < \infty$, $\lim_{x \to -\infty}\phi_0(x) = \omega_n$
and $\lim_{x \to +\infty}\phi_0(x) = \omega_{n+1}$.
Letting $x_1 \to -\infty$ and $x_2 \to \infty$ in \eqref{eq:Gamma-ineq}, we obtain
\begin{equation}
\label{eq:bogom-coer}
E_p(\phi_0) \geq \Gamma(\omega_{n+1}) - \Gamma(\omega_n) = \int_{\omega_n}^{\omega_{n+1}} \sqrt{2W(y)}\ud y.
\end{equation}
Inspecting the proof of \eqref{eq:Gamma-ineq}, we see that equality holds if $\phi_0 = H_{n, n+1}$, so that
\begin{equation}
\label{eq:V-of-H}
E_p(H_{n, n+1}) = \int_{\omega_n}^{\omega_{n+1}} \sqrt{2W(y)}\ud y,\qquad E_p(\phi) \geq E_p(H_{n,n+1}).
\end{equation}
The case $\lim_{x \to -\infty}\phi(x) = \omega_{n+1}$
and $\lim_{x \to +\infty}\phi(x) = \omega_{n}$ is analogous, and we obtain \eqref{eq:ground-kink}.

Let $n, n' \in I$ such that $|n - n'| = 1$.
Writing $\phi = H_{n, n'} + v$, we have the Taylor expansion
\begin{equation}
E_p(v) = E_p(H_{n, n'}) + \frac 12\la v, L_{n, n'}v\ra,
\end{equation}
where (here and later) $\la \cdot, \cdot\ra$ denotes the $L^2(\bR)$ inner product and
\begin{equation}
\label{eq:Lnn'-def}
L_{n, n'} := -\partial_x^2 + W''(H_{n,n'}) = -\partial_x^2 + V_{n, n'},
\end{equation}
with $V_{n, n'}(x) := W''(H_{n,n'}(x))$. Note that $|V_{n, n'}(x) - m_n^2| \lesssim \eee^{m_n x}$ as $x \to -\infty$,
and $|V_{n, n'}(x) - m_{n'}^2| \lesssim \eee^{-m_{n'}x}$ as $x \to \infty$.

Spectral information on $L_{n, n'}$ is obtained using standard arguments.
For the convenience of the Reader, we state without proof a minor modification of Lemma~2.6 from \cite{BKMM}.
\begin{proposition}\cite[Lemma~2.6]{BKMM}
\label{prop:L-index}
The operator $L_{n, n'}$ defined by \eqref{eq:Lnn'-def}
has the following properties:
\begin{enumerate}
\item $L_{n, n'}$ is self-adjoint on $L^2(\bR)$, with domain $H^2(\bR)$,
\item $\ker(L_{n, n'}) = \tx{span}(\partial_x H_{n, n'})$ and $\tx{spec}(L) \subset \{0\} \cup [\lambda, +\infty)$ for some $\lambda > 0$,
\item for any $\cZ \in L^2(\bR)$ such that $\la \cZ, \partial_x H_{n,n'}\ra \neq 0$
there exists $\lambda_0 > 0$ such that
for all $g \in H^1(\bR)$
\begin{equation}
\label{eq:vLv-coer}
\la g, L_{n,n'}g\ra \geq \lambda_0 \|g\|_{H^1}^2 - \frac{1}{\lambda_0}\la \cZ, g\ra^2.
\end{equation}
\end{enumerate}
\qed
\end{proposition}
\subsection{Coercivity near a moving kink}
The Lorentz transformation corresponding to the velocity $v\in (-1,1)$
followed by a shift in the $x$ direction by $a\in \bR$ is given by
\begin{equation}
\label{eq:lorentz-1}
\begin{aligned}
x'&=\gamma(x-a-vt),\\
t'&=\gamma(t-v(x-a)),
\end{aligned}
\qquad \gamma:=(1-v^2)^{-1/2}.
\end{equation}
In the sequel, given space-time variables $(t,x)$, we will use $(t',x')$ to denote their transform as above.
The inverse transform is
\begin{equation}
\begin{aligned}
x&=a+\gamma(x'+vt'),\\
t&=\gamma(t'+vx').
\end{aligned}
\end{equation}

Let $n, n'$ satisfy $|n - n'| = 1$, $-1 < v < 1$ and $a \in \bR$.
We would like to understand the linearised flow around the moving kink solution
\begin{equation}
\phi(t, x) := H_{n, n'}(x') = H_{n, n'}(\gamma(x - vt - a)).
\end{equation}

It is convenient to formulate \eqref{eq:csf} and its linearisation as first order in time systems.
If $\bs \phi = (\phi, \dot \phi)$, then \eqref{eq:csf} can be written as
\begin{equation}
\label{eq:csf-ham}
\partial_t \bs \phi(t, x) = \bs J \vD E(\bs \phi(t, x))
\end{equation}
where
\begin{equation}
\bs J:=\begin{pmatrix}
0 & 1\\
-1 & 0\end{pmatrix}, \qquad \vD E(\bs \phi) := \begin{pmatrix}
-\partial_{x}^2\phi + W'\left(\phi\right) \\
\dot{\phi}
\end{pmatrix}
\end{equation}
are the standard symplectic matrix and the Fr\'echet derivative of the energy.
The linearisation around the moving kink $H_{n,n'}(\gamma(x-vt-a))$ is
\begin{equation}
\label{eq:csf-ham-one-lin}
\partial_t \bs h(t, x) =\bs J\bs L_{v,n,n'}\left(vt+a\right)\bs h(t, x)
\end{equation}
where
\begin{equation}
\bs L_{v,n,n'}(a)=\begin{pmatrix}
-\partial_{x}^2+V_{n,n'}(\gamma(\cdot-a)) & 0\\
0 & 1
\end{pmatrix}.
\end{equation}

%

The one-dimensional kernel of $L_{n,n'}$ induces a
two-dimensional iterated kernel of $\bs L_{v, n, n'}$. We define
\begin{align}\label{eq:movingmodes}
Y_{n,n'}^{0}(v; x) &:= \left(\partial_x H_{n, n'}(\gamma x), -\gamma v\partial_{x}^2 H_{n,n'}(\gamma x)\right),\\
Y_{n,n'}^{1}(v; x) &:=\left(-vx\partial_x H_{n,n'}(\gamma x),\gamma\partial_x H_{n,n'}(\gamma x)+\gamma v^{2}x\partial_{x}^2 H_{n,n'}(\gamma x)\right),\\
\psi_{n, n'}^{0}(v; x) &:= \bs J Y_{n,n'}^{0}(v;x), \\
\psi_{n, n'}^{1}(v; x) &:= \bs J Y_{n,n'}^{1}(v;x).
\end{align}


The significance of these objects for the dynamics of \eqref{eq:csf-ham-one-lin}
is explained by the next lemma, which we state without proof.
\begin{lemma}
	\label{lem:kg-one-comp1}
	The following functions are solutions of \eqref{eq:csf-ham-one-lin}:
	\begin{align}
	\bs h(t, x) &= Y_{n, n'}^{0}(v; x -v t-a ),\label{eq:part-sol-3} \\
	\bs h(t, x) &= Y_{n, n'}^{1}(v; x -v t-a ) + \gamma((1 - v^2)t - va)Y_{n, n'}^{0}(v; x - v t - a ). \label{eq:part-sol-4}
	\end{align}
	If $\bs h(t, x)$ is any solution of \eqref{eq:csf-ham-one-lin}, then
	\begin{align}
	\dd t\la \psi_{n, n'}^{0}(v;\cdot - v t-a ), \bs h(t)\ra &= 0,\label{eq:covector-3}  \\
	\dd t\la \psi_{n, n'}^{1}(v;\cdot - v t-a ), \bs h(t)\ra &= -\frac{1}{\gamma}\la \psi_{n, n'}^{0}(v; \cdot - vt - a ), \bs h(t)\ra.\label{eq:covector-4}
	\end{align}
\end{lemma}
\begin{proof} See \cite[Lemma 4.3]{CJ1-19}.
Here, $\la\cdot, \cdot \ra $ denotes the inner product in $L^2(\bR) \times L^2(\bR)$.
\end{proof}




The quadratic form associated to the linear equation \eqref{eq:csf-ham-one-lin} above is
\[
Q_{v,n,n'}\left(a;\bm{h}_{0},\bm{h}_{0}\right)=\frac{1}{2}\int\left((\dot{h}_{0})^{2}+2v\dot{h}_{0}\partial_{x}h_{0}+(\partial_{x}h_{0})^{2}+V_{n,n'}\left(\gamma\left(\cdot-a\right)\right)\right)h_{0}^{2}\,dx.
\]
By the Lorentz transform, we can translate Proposition \ref{prop:L-index} to a coercivity property around $\bs L_{v,n,n'}$,
which yields the following result.
\begin{proposition}\label{prop:coermoving}
	For any $n, n'$ such that $|n - n'| = 1$ and $-1 < v < 1$, there exists $\lambda_0>0$ such that for
	all $\bm{h}_{0}$ and $a \in\mathbb{R}$ the following bound holds:
	\begin{align*}
	Q_{v,n,n'}\left(a;\bm{h}_{0},\bm{h}_{0}\right) \geq \lambda_0\left\Vert \bm{h}_{0}\right\Vert _{\cE}^{2}
	-\frac{1}{\lambda_0}\left(\left\langle \psi_{n,n'}^{0}\left(v;\cdot-a\right),\bm{h}_{0}\right\rangle ^{2}+\left\langle \psi_{n,n'}^{1}\left(v;\cdot-a\right),\bm{h}_{0}\right\rangle ^{2}\right).
	\end{align*}
\end{proposition}
\begin{proof}
For a proof in a slightly more general setting, see \cite[Lemma 4.4]{CJ1-19}, or the earlier work \cite[Proposition 1]{CoMu2014} for a different proof of a very similar result.
\end{proof}

\subsection{Coercivity near a multikink}
\label{ssec:coercivity}

In this subsection, we collect the coercivity properties for the linear operator around a multikink. Let $K \in \bN$,
$\bs n$ be a chain of vacua, $\bs v \in S^{(K)}$ and $\bs a \in \bR^K$. We set $\gamma_j := (1 - v_j^2)^{-1/2}$ for $1 \leq j \leq K$.
In general, various constants in the estimates below depend on $\bs v$ and $\bs a$,
but can be chosen uniformly for $(\bs v, \bs a) \in A \ssubset S^{(K)}\times \bR^K$.

We are interested in  the linear equation
\begin{equation}\label{eq:multikink}
\partial_t \bm{h}(t, x)=\bs J\bs L(\bs v, \bs a; t, x)\bm{h}(t, x),
\end{equation}
where
\begin{equation}
\label{eq:bsL-def}
\bs L(\bs v, \bs a; t, x)=\begin{pmatrix}
-\partial_{x}^2 + V(\bs v, \bs a; t, x) & 0 \\
0 & 1
\end{pmatrix}
\end{equation}
with
\begin{equation}\label{eq:V(t)}
V(\bs v, \bs a; t, x) := V_{n_0, n_1}(\gamma_1(x -v_1 t - a_1))
+ \sum_{j=1}^{K-1}\big(V_{n_j, n_{j+1}}(\gamma_{j+1}(x -v_{j+1} t - a_{j+1})) - m_{n_j}^2\big).
\end{equation}

Using notations from \eqref{eq:movingmodes}, for $1 \leq j \leq K$ we define
\[
Y_{j}^{0}(t, x):=Y_{n_{j-1},n_j}^{0}\left(v_j;x-v_{j}t-a_j\right),\qquad Y_{j}^{1}(t, x):=Y_{n_{j-1},n_j}^{1}\left(v_j;x-v_{j}t-a_j\right),
\]
and
\[
\psi_{j}^{0}(t, x):=\bs JY_{j}^{0}(t, x), \qquad \psi_{j}^{1}(t, x):=\bs JY_{j}^{1}(t, x).
\]


Lemma~\ref{lem:kg-one-comp1}, combined with the exponential decay from Proposition~\ref{prop:decaykink}, leads to the following result.
\begin{proposition}\label{prop:msol}
For any chain of vacua $\bs n$, $\bs v \in S^{(K)}$ and $\bs a \in \bR^K$
there exist $T_0 \in \bR $ and $\eta > 0$ such that if $\bm{h}(t, x)$ is a solution of \eqref{eq:multikink},
then for all $t \geq T_0$
	\begin{equation}
\label{eq:zero-mode-0}
	\Big|\dd t\left\langle \psi_{j}^{0}(t),\bm{h}(t)\right\rangle\Big| \lesssim e^{-\eta t}\left\Vert \bm{h}(t)\right\Vert _{\cE},
	\end{equation}
	\begin{equation}
\label{eq:zero-mode-1}
	\Big|\dd t\left\langle \psi_{j}^{1}(t),\bm{h}(t)\right\rangle +\frac{1}{\gamma_{j}}\left\langle \psi_{j}^{0}(t),\bm{h}(t)\right\rangle \Big|\lesssim e^{-\eta t}\left\Vert \bm{h}(t)\right\Vert _{\cE}.
	\end{equation}
\end{proposition}
\begin{proof}
See the proof of Lemma 4.8 in \cite{CJ1-19}.
\end{proof}
\begin{remark}
Since, for any $\bs v, \bs a$ and $t$, $\bs L(\bs v, \bs a; t)$ is a self-adjoint operator
on $L^2(\bR)\times L^2(\bR)$ and $\bs J$ is a skew-adjoint operator on the same space,
\eqref{eq:zero-mode-0} can be equivalently written
\begin{equation}
\big|\la\partial_t \psi_j^0(t), \bs h(t)\ra + \la \psi_j^0(t), \bs J\bs L(\bs v, \bs a; t)\bs h(t)\ra\big| = \big|\la \partial_t \psi_j^0(t) - \bs J\bs L(\bs v, \bs a; t)\psi_j^0(t), \bs h(t)\ra \big| \lesssim \eee^{-\eta t}\|\bs h(t)\|_\cE,
\end{equation}
in other words
\begin{equation}
\label{eq:zero-mode-01}
\|\partial_t \psi_j^0(t) - \bs J\bs L(\bs v, \bs a; t)\psi_j^0(t)\|_{\cE^*} \lesssim \eee^{-\eta t}.
\end{equation}
Analogously, \eqref{eq:zero-mode-1} can be equivalently written
\begin{equation}
\label{eq:zero-mode-11}
\|\partial_t \psi_j^1(t) - \bs J\bs L(\bs v, \bs a; t)\psi_j^1(t) + \gamma_j^{-1}\psi_j^0(t)\|_{\cE^*} \lesssim \eee^{-\eta t}.
\end{equation}
\end{remark}

Following \cite{CoMu2014, CoMa-multi}, as well as earlier papers \cite{MMT02, Martel05} in the KdV setting,
we construct a quadratic functional allowing to control the error
in a neighborhood of a multi-soliton. The point is that the quadratic form
$Q_{v,n,n'}\left(a;\bm{h}_{0},\bm{h}_{0}\right)$ defined above
contains a term depending on $v$. In order to study a neighbourhood of a multikink,
we need a functional which looks like $Q_{v,n_{j-1},n_j}\left(v_j t + a_j;\bm{h}_{0},\bm{h}_{0}\right)$
near the $j$-th kink.
Set
\begin{equation}
\label{eq:cut-off}
\chi_{j}\left(t,x\right):=\chi\left(\frac{x-v_{j}t-a_{j}}{\rho t}\right),
\end{equation}
where $\rho$ is a small positive number, $\chi$ is a smooth bump function such that $\chi(x)=1$ for $x\in\left[-1,1\right]$
and $\chi(x)=0$ for $x\in\mathbb{R}\backslash\left[-2,2\right]$.

Using the notation \eqref{eq:V(t)}, we define the quadratic form
\begin{equation}
\label{eq:Qmulti}
Q_{\bs v, \bs a}\left(t;\bm{h}_{0},\bm{h}_{0}\right)=\frac{1}{2}\int_\bR\bigg((\dot{h}_{0})^{2}+(\partial_{x}h_{0})^{2}+2\sum_{j=1}^{K}\chi_{j}(t)v_j\dot{h}_{0}\partial_{x}h_{0}+V(t)h_{0}^{2}\bigg)\ud x.
\end{equation}
\begin{proposition}\label{prop:coermulti}
Fix $K \in \bN$ and a chain of vacua $\bs n$.
For any $A \ssubset S^{(K)}\times \bR^K$ there exist $T_0 \in \bR $ and $\lambda_0 > 0$
such that for all $(\bs v, \bs a) \in A$, $\bm{h}_{0}\in \cE$ and $t \geq T_0$
	\begin{align*}
	Q_{\bs v, \bs a}\left(t;\bm{h}_{0},\bm{h}_{0}\right) \geq \lambda_0\left\Vert \bm{h}_{0}\right\Vert _{\cE}^{2}
	-\frac{1}{\lambda_0}\sum_{j=1}^{K}\left(\left\langle \psi_{j}^{0}(t),\bm{h}_{0}\right\rangle ^{2}+\left\langle \psi_{j}^{1}(t),\bm{h}_{0}\right\rangle ^{2}\right).
	\end{align*}
\end{proposition}
\begin{proof}
See \cite[Lemma 4.6]{CJ1-19}.
\end{proof}

\section{Proofs of the main results}
\label{sec:fixed}

\subsection{Function spaces}
\label{ssec:spaces}
For given $T \in \bR$, we denote $C_0^\infty([T, \infty)\times \bR)$
the set of functions $[T, \infty) \times \bR \to \bR$ which are restrictions to $[T, \infty)\times \bR$
of functions in $C_0^\infty(\bR^2)$.
Similarly, if $A \ssubset S^{(K)} \times \bR^K$ is an open set with compact closure,
we denote $C_0^\infty(\conj A \times [T, \infty)\times \bR)$
the set of functions $\conj A\times [T, \infty) \times \bR \to \bR$ which are restrictions to $\conj A \times [T, \infty)\times \bR$
of functions in $C_0^\infty(\bR^{2K + 2})$.

Fix $K \in \bN$, a chain of vacua $\bs n$ and $(\bs v, \bs a) \in S^{(K)}\times \bR^K$.
For $T \geq 0$, $\delta > 0$ and $s \in \bN$,
we consider the space $\cH^s_{T, \delta}$ of functions $g: (T, \infty) \times \bR \to \bR$,
which is defined as the completion of the set $C_0^\infty([T, \infty)\times \bR)$ for the norm
\begin{equation}
\|g\|_{\cH^s_{T, \delta}}^2 := \sup_{t > T}\ \eee^{2\delta t}\int_\bR\sum_{r+u\leq s}(\partial_t^r \partial_x^u g(t, x))^2\ud x.
\end{equation}
In other words, $\cH^s_{T, \delta}$ is the space of functions whose partial derivatives of order $< s$
belong to $\cH_{T, \delta} = \cH_{T, \delta}^1$. The definition also makes sense for $s = 0$ and gives
a weighted $L^\infty L^2$ space, which we denote $L^2_{T, \delta} = \cH_{T, \delta}^0$.
If $b \in L^\infty((T, \infty)\times \bR)$, then
\begin{equation}
\label{eq:bh-L2}
\|bh\|_{L^2_{T, \delta}} \leq \|b\|_{L_{t, x}^\infty}\|h\|_{L^2_{T, \delta}},\qquad\text{for all }b, h: (T, \infty)\times \bR \to \bR.
\end{equation}
By the Chain Rule and the embedding $H^1(\bR) \subset L^\infty(\bR)$, we also obtain
\begin{equation}
\label{eq:bh-E}
\|bh\|_{\cH_{T, \delta}} \lesssim \|b\|_X \|h\|_{\cH_{T, \delta}}.
\end{equation}
where the norm $X$ is defined as
\begin{equation}
\label{eq:X-def}
\|b\|_X := \|b\|_{L_{t, x}^\infty} + \|\partial_t b\|_{L_t^\infty L_x^2} +
\|\partial_x b\|_{L_t^\infty L_x^2}.
\end{equation}
It is clear that $X$ is an algebra.
Finally, again using $H^1(\bR) \subset L^\infty(\bR)$, we have
\begin{equation}
\label{eq:gh-E}
\|gh\|_{\cH_{T, \delta_1+\delta_2}} \lesssim \|g\|_{\cH_{T, \delta_1}} \|h\|_{\cH_{T, \delta_2}}.
\end{equation}

Note that, if $\rho_\epsilon$ is a standard smooth space-time approximation of identity and $g \in \cH^s_{T,\delta}$,
then $\lim_{\epsilon \to 0^+}\|\rho_\epsilon \ast g - g\|_{\cH^s_{T, \delta}} = 0$. 

If, instead of $(\bs v, \bs a)$ being fixed, they are allowed to vary
in an open set with compact closure $A \ssubset S^{(K)} \times \bR^K$, we consider the space $C^k(A, \cH^s_{T, \delta})$
of functions $g: A \times (T, \infty) \times \bR \to \bR$,
which is defined as the completion of the set $C_0^\infty(\conj A \times [T, \infty)\times \bR)$, for the norm
\begin{equation}
\|g\|_{C^k(A, \cH^s_{T, \delta})}^2 := \sup_{(\bs v, \bs a) \in A}\sup_{t > T}\ \eee^{2\delta t}\int_\bR\sum_{|\bs p| + |\bs q| \leq k}\sum_{r+u\leq s}\big(\partial_{\bs v}^{\bs p} \partial_{\bs a}^{\bs q} \partial_t^r \partial_x^u g(\bs v, \bs a, t, x)\big)^2\ud x,
\end{equation}
in other words the space of functions whose partial derivatives of order $< s$ belong to $\cH_{T, \delta}$,
and depend in a $C^k$ way on the parameters $(\bs v, \bs a) \in A$.
Here, $\bs p, \bs q \in \bN^K$ are multi-indices.

Note that, if $g \in \cH_{T, \delta}^s$, then
\begin{equation}
\label{eq:conv-to-0}
\lim_{t \to \infty}\eee^{2\delta t}\int_\bR\sum_{r+u\leq s}(\partial_t^r \partial_x^u g(t, x))^2\ud x = 0.
\end{equation}
Similarly, if $g \in C^k(A, \cH_{T, \delta}^s)$, then
\begin{equation}
\label{eq:param-conv-to-0}
\lim_{t \to \infty}\sup_{(\bs v, \bs a) \in A}\eee^{2\delta t}\int_\bR\sum_{|\bs p| + |\bs q| \leq k}\sum_{r+u\leq s}(\partial_{\bs v}^{\bs p} \partial_{\bs a}^{\bs q} \partial_t^r \partial_x^u g(\bs v, \bs a, t, x))^2\ud x = 0.
\end{equation}

If $E$ is a normed vector space, we denote $B(E)$ the unit ball in $E$.

In the sequel, we adopt the convention of dropping arguments of functions from right to left,
that is, if $f: X \times Y \to Z$ is a function, then for $x \in X$, $f(x)$ denotes the function $Y \to Z$
given by $f(x)(y) := f(x, y)$, and similarly for more arguments,
thus for example the symbole $H$ can designate the function $(\bs v, \bs a; t, x) \mapsto H(\bs v, \bs a; t, x)$.

%

\subsection{Energy estimates}

For given $\bs v = (v_1, \ldots, v_K)\in S^{(K)}$ and $\bs a = (a_1, \ldots, a_K) \in \bR^K$,
we seek $g = g(\bs v, \bs a) \in \cH_{T, \delta}^s$ such that
\begin{equation}
\label{eq:decomp}
\phi := H(\bs v, \bs a) + g
\end{equation}
solves \eqref{eq:csf}.
In the next lemma, we prove energy estimates for the linearised problem.

\begin{lemma}
\label{lem:nonhom}
For all $K \in \bN$, a chain of vacua $\bs n$, $A \ssubset S^{(K)}\times \bR^K$ and $\delta_0$ sufficiently small
(depending on $A$), there exists $T_0$ such that if $T \geq T_0$,
then for any $(\bs v, \bs a) \in A$ and $f \in L^2_{T, \delta_0}$
there exists a~unique solution $R(\bs v, \bs a)f := h\in \cH_{T, \delta_0}$ of the equation 
\begin{equation}\label{eq:lemma31}
\partial_t^2 h - \partial_x^2 h + V(\bs v, \bs a)h = f,\qquad f, h: (T, \infty)\times \bR \to \bR.
\end{equation}
Moreover, $R(\bs v, \bs a)$ is a bounded operator $L_{T, \delta_0}^2 \to \cH_{T, \delta_0}$,
whose norm stays bounded when $T \geq T_0$ and $(\bs v, \bs a) \in A$.
\end{lemma}
\begin{proof}

\noindent\textbf{Step 1} (a priori estimates).
Let $f \in L^2_{T, \delta}$ and assume $h \in \cH_{T, \delta}$ is a weak solution of \eqref{eq:lemma31}.
We will show that
\begin{equation}
\label{eq:a-priori}
\|h\|_{\cH_{T, \delta}} \leq C_\delta \|f\|_{L^2_{T, \delta}},
\end{equation}
where $C_\delta$ depends on $\delta$, but not on $T, \bs v$ or $\bs a$.
Since, in this proof, $\bs v$ and $\bs a$ are considered as fixed, we will write $V$ instead of $V(\bs v, \bs a)$.

Without loss of generality, we can assume $h$ and $f$ are smooth functions.
Indeed, if $\rho_\epsilon$ is a standard smooth approximation of identity in dimension $2$, then
\begin{equation}
\partial_t^2 (\rho_\epsilon \ast h) - \partial_x^2 (\rho_\epsilon\ast h) + V(\rho_\epsilon\ast h)
= \rho_\epsilon\ast f - \rho_\epsilon\ast(V h) + V(\rho_\epsilon\ast h).
\end{equation}
If \eqref{eq:a-priori} holds for smooth functions, then we obtain
\begin{equation}
\label{eq:a-priori-conv}
\|\rho_\epsilon\ast h\|_{\cH_{T, \delta}} \leq C_\delta \|\rho_\epsilon\ast f - \rho_\epsilon\ast(V h) + V(\rho_\epsilon\ast h)\|_{L^2_{T, \delta}},
\end{equation}
for all $\epsilon > 0$. Both $\rho_\epsilon\ast(V h)$ and $V\rho_\epsilon\ast h$ converge to $V h$
in $L^2_{T, \delta}$ as $\epsilon \to 0$, hence we get \eqref{eq:a-priori} in the limit.

We now prove \eqref{eq:a-priori} for smooth functions by means of an energy estimate, that is,
we compute the time-derivatives of $Q\left(t;\bm{h}(t),\bm{h}(t)\right)$ and of the projections of $\bs h$ on the iterated kernel.
One additional difficulty is caused by the fact that $V$ is not a sum of localised terms, and approaches distinct
values $m_{n_0}^2, m_{n_1}^2, \ldots, m_{n_K}^2$ in the regions between the kinks.
In order to deal with this issue, we decompose $V$ in the following way.

Let $x_0 > 0$ to be chosen later. Let $\wt V_1$ be a smooth function such that $\wt V_1(x) = m_{n_0}^2$ for $x \leq -x_0$
and $\wt V_1(x) = m_{n_1}^2$ for $x \geq x_0$. For $2 \leq j \leq K$, let $\wt V_j$ be a smooth function such that
$\wt V_j(x) = 0$ for $x \leq -x_0$ and $\wt V_j(x) = m_{n_j}^2 - m_{n_{j-1}}^2$ for $x \geq x_0$.
We require $\|\partial_x \wt V_j\|_{L^\infty} \lesssim x_0^{-1}$ for all $1\leq j \leq K$.
We set
\begin{equation}
V_1(x) := V_{n_0, n_1}(\gamma_1 x) - \wt V_1
\end{equation}
and, for $2 \leq j \leq K$,
\begin{equation}
V_j(x) := V_{n_{j-1}, n_j}(\gamma_j x) - m_{n_{j-1}}^2 - \wt V_j.
\end{equation}
Then $V_1, \ldots, V_K$ are smooth exponentially decaying functions and, see \eqref{eq:V(t)},
\begin{equation}
V(t, x) = \sum_{j=1}^K \wt V_j(x - v_j t - a_j) + \sum_{j=1}^K V_j(x - v_j t - a_j).
\end{equation}

In the remaining part of this proof, we call ``negligible'' the terms which can be made
$\ll \|\bs h(t)\|_\cE^2 + \|\bs h(t)\|_\cE\|f(t)\|_{L^2}$ upon taking $T_0$ and $x_0$ large enough.
The sign $\simeq$ means equality up to negligible terms.
We have
\begin{align*}
\dd t Q\left(t;\bm{h}(t),\bm{h}(t)\right) & =\int_\bR\big(\partial_t h\partial_t^2{h}+\partial_{x}h\partial_{t}\partial_x{h}+V h\partial_t{h}\big)\ud x\\
& +\dd t\int_\bR\sum_{j=1}^{J}\chi_{j}(t)\partial_t{h}v_{j}\partial_{x}h\ud x +\frac 12\int_\bR h^{2}\partial_{t}V\ud x.
\end{align*}
Note that
\begin{align*}
\int_\bR h^{2}\partial_{t}V\ud x & =\int_\bR h^{2}\partial_{t}\bigg(\sum_{j=1}^K \wt V_j(\cdot - v_j t - a_j) + \sum_{j=1}^K V_j(\cdot - v_j t - a_j)\bigg)\ud x\\
& =-\sum_{j=1}^K \int_\bR h^2 v_j \partial_x \wt V_j(\cdot - v_j t - a_j)\ud x - \sum_{j=1}^K \int_\bR h^2 v_j \partial_x V_j(\cdot - v_j t - a_j)\ud x.
\end{align*}
Since $\|\partial_x \wt V_j\|_{L^\infty} \lesssim x_0^{-1}$, the first sum is negligible.
Integrating by parts the terms in the second sum, we obtain
\begin{equation}
\frac 12 \int_\bR h^{2}\partial_{t}V\ud x \simeq \sum_{j=1}^K \int_\bR v_j h\partial_{x}h V_{j}(\cdot-v_{j}t-a_{j})\ud x.
\end{equation}
We also have
\begin{align}\label{eq:dtXi}
\dd t\int\sum_{j=1}^{K}\chi_{j}\partial_t{h}v_{j}\partial_{x}h\ud x & =\sum_{j=1}^{K}\int_\bR\partial_{t}\chi_{j}\partial_t{h}v_{j}\partial_{x}h\ud x\\
& +\sum_{j=1}^{K}\int_\bR\chi_{j}\partial_t^2{h}v_{j}\partial_{x}h\ud x\\
& +\sum_{j=1}^{K}\int_\bR\chi_{j}\partial_t{h}v_{j}\partial_{t}\partial_x{h}\ud x.
\end{align}
Since $\|\partial_t \chi_j\|_{L^\infty} \lesssim t^{-2}$, the first term of the right hand side is negligible.
In the second term of the right-hand side above, we replace $\partial_t^2{h}$ by $\partial_{x}^2h-h-Vh+f$.
Integrating by parts, we see that $\partial_x^2 h$ and $h$ yield negligible terms
(whenever the differentiation hits $\chi_{j}$,
it produces a $\frac{1}{t}$ factor). By the same argument, the third term of the right hand side of \eqref{eq:dtXi} is negligible.
Consider the term
\begin{equation}
\int_\bR v_j\chi_{j}Vh \partial_{x}h\ud x = \int_\bR v_j\chi_{j}h\partial_x h\bigg(\sum_{i=1}^K \wt V_j(\cdot - v_j t - a_j) + \sum_{i=1}^K V_j(\cdot - v_j t - a_j)\bigg)\ud x.
\end{equation}
Since $V_j$ are exponentially decaying, the terms $V_i$ for $i \neq j$ can be neglected.
Integrating by parts, we see that the terms $\wt V_i$ can be neglected as well, and we are left with
\begin{equation}
\int_\bR v_j\chi_{j}Vh \partial_{x}h\ud x \simeq \int_\bR v_j h\partial_{x}h V_{j}(\cdot-v_{j}t-a_{j})\ud x.
\end{equation}
Note that in the third term of \eqref{eq:dtXi}, we perform integration by parts
\[
\int\sum_{j=1}^{J}\chi_{j}\partial_t{h}v_{j}\partial_{t}\partial_x{h}\ud x=-\frac{1}{2}\int\sum_{j=1}^{J}\partial_{x}\chi_{j}v_{j}(\partial_t{h})^{2}\ud x.
\]
Combining the estimates above, we obtain
\begin{equation}\label{eq:quadratic1}
\Big|\dd t Q\left(t;\bm{h}(t),\bm{h}(t)\right)\Big|\lesssim c_0\left\Vert \bm{h}(t)\right\Vert _{\mathcal{E}}^{2}+\left\Vert f(t)\right\Vert _{L^{2}}\left\Vert \bm{h}(t)\right\Vert _{\mathcal{E}},
\end{equation}
where $c_0$ can be made arbitrarily small by taking $T_0$ and $x_0$ large.

Next, we compute the time derivatives of the projections on the iterated kernel.
Using the notation from Section~\ref{ssec:coercivity}, we can rewrite \eqref{eq:lemma31} as
\begin{equation}
\partial_t \bs h(t) = \bs J\bs L(\bs v, \bs a; t)\bs h(t) + (0, f(t)).
\end{equation}
Using the self-adjointness of $\bs L(\bs v, \bs a; t)$ and the skew-adjointness of $\bs J$, we have
\begin{equation}
\begin{aligned}
\dd t\big\langle \psi_{j}^{0}(t),\bm{h}(t)\big\rangle &=\big\la \partial_t \psi_{j}^{0}(t),\bm{h}(t)\big\ra +\big\langle \psi_{j}^{0}(t),\partial_t \bm{h}(t)\big\rangle \\
&= \la \partial_t \psi_j^0(t) - \bs J\bs L(\bs v, \bs a; t)\psi_j^0(t), \bs h(t)\ra + \la\psi_j^0(t), (0, f(t))\ra.
\end{aligned}
\label{eq:zero1}
\end{equation}
Using \eqref{eq:zero-mode-01}, we infer
\begin{align*}
\Big|\dd t\left\langle \psi_{j}^{0}(t),\bm{h}(t)\right\rangle\Big| \lesssim e^{-\eta t}\left\Vert \bm{h}(t)\right\Vert _{\mathcal{E}}+\left\Vert f(t)\right\Vert _{L^{2}}.
\end{align*}
Similarly, \eqref{eq:zero-mode-11} yields
\[
\Big|\dd t\left\langle \psi_{j}^{1}(t),\bm{h}(t)\right\rangle\Big| \lesssim \big|\left\langle \psi_{j}^{0}(t),\bm{h}(t)\right\rangle\big|+e^{-\eta t}\left\Vert \bm{h}(t)\right\Vert _{\mathcal{E}}+\left\Vert f(t)\right\Vert _{L^{2}}.
\]
Therefore, integrating from $\infty$, for zero modes, we obtain that
\[
\left|\left\langle \psi_{j}^{0}(t),\bm{h}(t)\right\rangle \right|\lesssim\frac{e^{-\left(\eta+\delta\right)t}}{\eta+\delta}\left\Vert \bm{h}\right\Vert _{\cH_{T_0, \delta}}+\frac{1}{\delta}e^{-\delta t}\left\Vert f\right\Vert _{L_{T_0, \delta}^2}
\]
and
\[
\left|\left\langle \psi_{j}^{1}(t),\bm{h}(t)\right\rangle \right|\lesssim\left(\frac{e^{-\left(\eta+\delta\right)t}}{\left(\eta+\delta\right)^{2}}+\frac{e^{-\left(\eta+\delta\right)t}}{\eta+\delta}\right)\left\Vert \bm{h}\right\Vert _{\cH_{T_0, \delta}}+\left(\frac{1}{\delta}+\frac{1}{\delta^{2}}\right)e^{-\delta t}\left\Vert f\right\Vert _{L^{2}_{T_0, \delta}}.
\]
For the quadratic form part, using \eqref{eq:quadratic1} and the
bootstrap assumptions, we integrate from $\infty$ and conclude that
\[
\left|Q\left(t;\bm{h}(t),\bm{h}(t)\right)\right|\lesssim c_0\frac{e^{-2\delta t}}{2\delta}\left\Vert \bm{h}\right\Vert _{\cH_{T_0, \delta}}^{2}+\frac{e^{-2\delta t}}{2\delta}\left\Vert f\right\Vert _{L^{2}_{T_0, \delta}}\left\Vert \bm{h}\right\Vert _{\cH_{T_0, \delta}}.
\]
Therefore by the coercivity, one has
\begin{align*}
\left\Vert \bm{h}(t)\right\Vert _{\mathcal{E}}^{2} & \lesssim\left(\frac{e^{-\left(\eta+\delta\right)t}}{\left(\eta+\delta\right)^{2}}+\frac{e^{-\left(\eta+\delta\right)t}}{\eta+\delta}\right)^{2}\left\Vert \bm{h}\right\Vert _{\cH_{T_0, \delta}}^{2}+\left(\frac{1}{\delta}+\frac{1}{\delta^{2}}\right)^{2}e^{-2\delta t}\left\Vert f\right\Vert _{L^{2}_{T_0, \delta}}\\
& +c_0\frac{e^{-2\delta t}}{2\delta}\left\Vert \bm{h}\right\Vert _{\cH_{T_0, \delta}}^{2}+\frac{e^{-2\delta t}}{2\delta}\left\Vert f\right\Vert _{L^{2}_{T_0, \delta}}\left\Vert \bm{h}\right\Vert _{\cH_{T_0, \delta}},
\end{align*}
where $c_0$ can be made arbitrarily small by taking $T_0$ and $x_0$ large,
yielding \eqref{eq:a-priori}.

\noindent\textbf{Step 2} (existence and uniqueness of solutions).
Let $f \in L^2_{T, \delta}$.
Take $\wt T \geq T$ and let $f_{\wt T}(t) := \chi(t - \wt T)f(t)$,
where $\chi$ is a smooth decreasing function such that $\chi(t) = 1$ for $t \leq -1$ and $\chi(t) = 0$ for $t \geq 0$.
Then $\lim_{\wt T \to \infty}\|f_{\wt T} - f\|_{L^2_{T, \delta}} = 0$, see \eqref{eq:conv-to-0}.
Let $h_{\wt T}\left(t,x\right)$ be the solution of
\begin{equation}
\label{eq:hl-eq}
\partial_{t}^2h_{\wt T}-\partial_{x}^2 h_{\wt T}+V(t)h_{\wt T}=f_{\wt T}
\end{equation}
such that $\big({h}_{\wt T}(\wt T), \partial_t{h}_{\wt T}(\wt T)\big)= (0,0)$.
This solution exists by the standard local Cauchy theory.
Since $\wt f(t) = 0$ for all $t \geq \wt T$, we have $\wt h(t) = 0$ for all $t \geq \wt T$,
in particular $\wt h \in \cH_{T, \delta}$.

By Step 2, we have that $(h_{\wt T})_{\wt T}$ satisfies the Cauchy condition in $\cH_{T, \delta}$ as $\wt T \to \infty$,
hence it converges. Let $h \in \cH_{T, \delta}$ be the limit.
Passing to the limit in \eqref{eq:hl-eq}, we obtain that \eqref{eq:lemma31}
holds in the sense of distributions.

Uniqueness follows directly from the a priori estimate.
\end{proof}
In the next lemma, we are interested in the dependence on the parameters, thus $V$ should be seen again
as a real-valued function defined on $A\times [T_0, \infty) \times \bR$.
\begin{lemma}
\label{lem:reg-lin}
For $K \in \bN$, a chain of vacua $\bs n$, $A \ssubset S^{(K)}\times \bR^K$, $\delta_0$ sufficiently small
(depending on $A$) and $\delta_1 \in (0, \delta_0)$,
there exists $T_0$ such that if $l \in \bN_0, s \in \bN$, $T \geq T_0$,
then for any $b \in B(C^0(A, \cH_{T, \delta_0})) \cap C^l(A, \cH^s_{T, \delta_0})$,
$f \in C^l(A, \cH^{s-1}_{T, \delta})$
there exists a unique solution $h\in C^l(A, \cH^{s}_{T_0, \delta_1})$ of the equation 
\begin{equation}
\partial_t^2 h - \partial_x^2 h + (V+b)h = f, \qquad f, b, h: A \times (T_0, \infty) \times \bR \to \bR.
\end{equation}
\end{lemma}
\begin{proof}

\noindent
\textbf{Step 1} (existence of solutions).
Let $l = 0$ and $s = 1$.
For given $(\bs v, \bs a) \in A$, we rewrite the equation as follows:
\begin{equation}
\label{eq:above}
h(\bs v, \bs a) = R(\bs v, \bs a)\big(f(\bs v, \bs a) - b(\bs v, \bs a)h(\bs v, \bs a)\big).
\end{equation}
The embedding $H^1(\bR) \subset L^\infty(\bR)$ implies, by taking $T_0$ large enough, that $\|b\|_{L^\infty((T_0, \infty)\times \bR)}$ 
can be made as small as we wish.
Thus, by \eqref{eq:bh-L2} and Lemma~\ref{lem:nonhom}, \eqref{eq:above} defines a contraction in $\cH_{T_0, \delta_1}^1$.

\noindent
\textbf{Step 2} (continuity with respect to the parameters).
For each $(\bs v, \bs a) \in A$, let $h(\bs v, \bs a)$ be the function constructed in Step 1.
We will prove that $h \in C^0(A, \cH_{T_1, \delta_0}^1)$.
To this end, take $(\bs v, \bs a), (\sh{\bs v}, \sh{\bs a}) \in A$
and denote, for the sake of brevity, $\sh h := h(\sh{\bs v}, \sh{\bs a})$, $\sh f := f(\sh{\bs v}, \sh{\bs a})$,
$\sh V := V(\sh{\bs v}, \sh{\bs a})$ and $\sh b := b(\sh{\bs v}, \sh{\bs a})$, as well as $\wt h := h(\bs v, \bs a)$,
$\wt f := f({\bs v}, {\bs a})$,
$\wt V := V({\bs v}, {\bs a})$ and $\wt b := b({\bs v}, {\bs a})$. We have
\begin{equation}
\partial_t^2 (\sh h - \wt h) - \partial_x^2 (\sh h - \wt h) + V(\sh h - \wt h) = \sh f - \wt f - (\sh V - \wt V)\sh h - \sh b(\sh h - \wt h) + (\sh b - \wt b)\wt h.
\end{equation}
By the a priori estimate, it suffices to check that
\begin{equation}
\lim_{(\sh{\bs v}, \sh{\bs a}) \to (\bs v, \bs a)}\big(\|\sh f - \wt f\|_{L^2_{T_0,\delta_0}} + \|(\sh V - \wt V)\sh h\|_{L^2_{T_0,\delta_0}} + \|(\sh b - \wt b)\wt h\|_{L^2_{T_0,\delta_0}}\big) = 0
\end{equation}
(indeed, since $\|\sh b\|_{L_{t, x}^\infty}$ is small, the term $\sh b(\sh h - \wt h)$ can be absorbed).
As for $\sh f - \wt f$, this is clear.
Continuity of $b$ implies $\|(\sh b - \wt b)\wt h\|_{L^2_{T_0,\delta_0}} \to 0$.
Finally, $\sh V - \wt V$ is bounded and converges to $0$ locally in $t$, so \eqref{eq:param-conv-to-0}
implies the desired bound.

\noindent
\textbf{Step 3} (differentiability).
We prove regularity of $h$ by induction on $s$, and then on $l$.
First, we fix $l = 0$.
If $s = 1$, there is nothing to do. Let $s \geq 2$.
We prove existence of the partial derivatives.
Define $h_t$ as the solution of
\begin{equation}
\partial_t^2 h_t - \partial_x^2 h_t + (V+b)h_t = \partial_t f - (\partial_t V + \partial_t b)h.
\end{equation}
By the induction hypothesis, $h_t \in C^0(A, \cH_{T_0, \delta_1}^{s-1})$.
We prove that $\partial_t h = h_t$ in the weak sense.
For fixed $(\bs v, \bs a)$, we consider
\begin{equation}
r_\epsilon(t, x) := h(\bs v, \bs a; t+ \epsilon, x) - h(\bs v, \bs a; t, x) - \epsilon h_t(\bs v, \bs a; t, x).
\end{equation}
It satisfies the equation
\begin{equation}
\begin{aligned}
&\partial_t^2 r_\epsilon - \partial_x^2 r_\epsilon + (V + b)r_\epsilon = f_\epsilon - f - (V_\epsilon + b_\epsilon - V - b)h_\epsilon - \epsilon\partial_t f + \epsilon(\partial_t V + \partial_t b)h \\
&= (f_\epsilon - f - \epsilon \partial_t f) - (V_\epsilon + b_\epsilon - V - b - \epsilon(\partial_t V + \partial_t b))h_\epsilon
- \epsilon(\partial_t V + \partial_t b)(h_\epsilon -h),
\end{aligned}
\end{equation}
where we have denoted $V_{\epsilon}(\bs v, \bs a; t, x) := V(\bs v, \bs a; t + \epsilon, x)$,
and similarly for $f_\epsilon$, $h_\epsilon$ and $b_\epsilon$.
We need to check that the right hand side is $o(\epsilon)$ in $L^2_{T_0, \delta_1}$ as $\epsilon \to 0$.
For the first two terms, this follows from the differentiability of $V, b$ and $f$, and the uniform boundedness of $h_\epsilon$
in $L_{T_0, \delta_1}^2$.
For the third term, we use the continuity of $h$ with respect to $t$.
By the a priori bound, $\|r_\epsilon\|_{\cH_{T_0, \delta_1}} \ll \epsilon$, implying $\partial_t h = h_t$.

Analogously, one proves that $\partial_x h \in \cH_{T_0, \delta_1}^{s-1}$, thus $h \in \cH_{T_0, \delta_1}^s$,
finishing the induction step with respect to $s$.

The induction step with respect to $l$ is similar, so we only sketch the argument for $\partial_{v_k}h$.
Define $h_{v_k}$ as the solution of
\begin{equation}
\partial_t^2 h_{v_k} - \partial_x^2 h_{v_k} + (V+b)h_{v_k} = \partial_{v_k} f - (\partial_{v_k} V + \partial_{v_k} b)h.
\end{equation}
By the induction hypothesis, $h_{v_k} \in C^0(A, \cH_{T_0, \delta_1}^{s-1})$
(note that it is here that the loss in the exponential convergence rate occurs,
due to the polynomial growth of $\partial_{v_k}V$).
We prove that $\partial_{v_k} h = h_{v_k}$ in the weak sense.
For fixed $(\bs v, \bs a)$, we consider
\begin{equation}
r_\epsilon(t, x) := h(\bs v + \epsilon \bs e_k, \bs a; t, x) - h(\bs v, \bs a; t, x) - \epsilon h_{v_k}(\bs v, \bs a; t, x),
\end{equation}
where $\bs e_k := (0, \ldots, 1, \ldots, 0)$ is the $k$-th element of the standard basis of $\bR^K$.
It satisfies the equation
\begin{equation}
\begin{aligned}
&\partial_t^2 r_\epsilon - \partial_x^2 r_\epsilon + (V + b)r_\epsilon = f_\epsilon - f - (V_\epsilon + b_\epsilon - V - b)h_\epsilon - \epsilon\partial_{v_k} f + \epsilon(\partial_{v_k} V + \partial_{v_k} b)h \\
&= (f_\epsilon - f - \epsilon \partial_{v_k} f) - (V_\epsilon + b_\epsilon - V - b - \epsilon(\partial_{v_k} V + \partial_{v_k} b))h_\epsilon
- \epsilon(\partial_{v_k} V + \partial_{v_k} b)(h_\epsilon -h),
\end{aligned}
\end{equation}
where we have denoted $V_{\epsilon}(\bs v, \bs a; t, x) := V(\bs v+ \epsilon\bs e_k, \bs a; t , x)$,
and similarly for $f_\epsilon$, $h_\epsilon$ and $b_\epsilon$.
We need to check that the right hand side is $o(\epsilon)$ in $L^2_{T_0, \delta_1}$.
For the first two terms, this follows from the differentiability of $V, b$ and $f$,
and the uniform boundedness of $h_\epsilon$ in $L_{T_0, \delta_1}^2$.
For the third term, we know from the Step 2 that $\lim_{\epsilon \to 0}\|h_\epsilon - h\|_{L_{T_0, \delta_0}^2} = 0$,
hence $\lim_{\epsilon \to 0}\|(\partial_{v_k} V + \partial_{v_k} b)(h_\epsilon - h)\|_{L_{T_0, \delta_1}^2} = 0$ for any $\delta_1 < \delta_0$.
\end{proof}

\subsection{Formulation as a fixed point problem and conclusions}
In order to make the formulas which follow shorter, we denote
\begin{equation}
H_k(\bs v, \bs a; t, x) := H_{n_{k-1}, n_k}(\gamma_k(x - v_k t - a_k))
\end{equation}
the $k$-th kink.
Plugging the decomposition \eqref{eq:decomp} into \eqref{eq:csf}, we obtain the following equation for
the error term $g = g(\bs v, \bs a)$:
\begin{equation}
\label{eq:error-eq}
\begin{aligned}
\partial_t^2 g(\bs v, \bs a) - \partial_x^2 g(\bs v, \bs a) + V(\bs v, \bs a)g(\bs v, \bs a) = N(\bs v,\bs a; g(\bs v, \bs a)),
\end{aligned}
\end{equation}
where we define
\begin{equation}
\label{eq:N-def}
\begin{aligned}
&N(\bs v,\bs a; g)  := -W'(H(\bs v, \bs a)+g) + \sum_{k=1}^K W'(H_k(\bs v, \bs a)) + V(\bs v, \bs a)g.
\end{aligned}
\end{equation}
We are going to reformulate this equation as a fixed point problem for a contraction.



We need to study some properties of composition with a smooth function in the relevant functional spaces. Before stating the next lemma, we recall that $B(E)$ denotes the unit ball in a normed space~$E$.
\begin{lemma}
\label{lem:baby-cross}
Let $F \in C^\infty(\bR)$. For any $C_0 \geq 0$, $T_0 \geq 0$ and $\delta_0 > 0$ there exists $C_1 \geq 0$ such that
\begin{align}
\|F(\phi + g) - F(\phi)\|_{\cH_{T_0, \delta_0}} &\leq C_1\|g\|_{\cH_{T_0, \delta_0}}, \label{eq:taylor-1} \\
\|F(\phi + g) - F(\phi) - F'(\phi)g\|_{\cH_{T_0, 2\delta_0}} &\leq C_1 \|g\|_{\cH_{T_0, \delta_0}}^2. \label{eq:taylor-2}
\end{align}
 for all $g \in B(\cH_{T_0, \delta_0})$ and $\phi$ satisfying $\|\phi\|_X \leq C_0$.
\end{lemma}
\begin{proof}
The bound \eqref{eq:taylor-1} follows from \eqref{eq:taylor-2} and \eqref{eq:bh-E}.

In order to prove \eqref{eq:taylor-2}, we write
\begin{equation}
\label{eq:fund-thm}
F(\phi + g) - F(\phi) - F'(\phi)g = g^2\int_0^1 (1-\sigma)F''(\phi + \sigma g)\ud\sigma.
\end{equation}
By the Chain Rule we have
\begin{equation}
\|F''(\phi + \sigma g)\|_{L_{t, x}^\infty} + \|\partial_t (F''(\phi + \sigma g))\|_{L_t^\infty L_x^2} + \|\partial_x (F''(\phi + \sigma g))\|_{L_t^\infty L_x^2} \lesssim 1,
\end{equation}
so \eqref{eq:bh-E} and \eqref{eq:gh-E} yield the claim.
\end{proof}
\begin{lemma}
\label{lem:cross}
Let $F \in C^\infty(\bR)$.
For any $A \ssubset S^{(K)} \times \bR^K$ and $l, s \in \bN$, there exist $T_0 \in \bR$ and $\delta_0 > 0$ such that
\begin{equation}
\label{eq:cross}
F(H) -\Big[ F(\omega_{n_0}) + \sum_{k=1}^K \big(F(H_k) - F(\omega_{n_{k-1}})\big) \Big] \in B(C^l(A, \cH_{T_0, \delta_0}^s))
\end{equation}
and
\begin{equation}
\label{eq:cross-2}
(F(H) - F(H_j))\partial_x H_j\in B(C^l(A, \cH_{T_0, \delta_0}^s)), \qquad\text{for all }1 \leq j \leq K.
\end{equation}
Moreover, if $0 < \delta_1 \leq \delta_0$ and $g \in C^k(A, \cH_{T_0, \delta_1}^s)$, then
\begin{equation}
\label{eq:cross-3}
F(H + g) - F(H) \in C^k(A, \cH_{T_0, \delta}^s)
\end{equation}
for any $0 < \delta < \delta_1$.
\end{lemma}
\begin{proof}
\textbf{Step 1.}
We prove that, if $j < k$, then
\begin{align}
\label{eq:cross-step1}
(H_j - \omega_{n_j})(H_k - \omega_{n_{k-1}}) \in B(C^k(A, \cH_{T_0, \delta_0}^s)), \\
\label{eq:cross21-step1}
(H_j - \omega_{n_j})\partial_x H_k \in B(C^k(A, \cH_{T_0, \delta_0}^s)), \\
\label{eq:cross22-step1}
(H_k - \omega_{n_{k-1}})\partial_x H_j \in B(C^k(A, \cH_{T_0, \delta_0}^s)).
\end{align}
By Proposition~\ref{prop:decaykink}, together with the Chain and Leibniz Rules,
\begin{equation}
\sum_{|\bs p| + |\bs q| \leq l}\sum_{r + u \leq s}|\partial_{\bs v}^{\bs p}\partial_{\bs a}^{\bs q}\partial_t^r\partial_x^u (H_j(\bs v, \bs a; t, x) - \omega_{n_j})| \leq \eee^{-\delta(x - v_j t - a_j)_+}.
\end{equation}
Indeed, if $|\bs p| + |\bs q| + r + u > 0$ then $\partial_{\bs v}^{\bs p}\partial_{\bs a}^{\bs q}\partial_t^r\partial_x^u(H_j(\bs v, \bs a; t, x) - \omega_{n_j})$
is the sum of a finite number of terms, each of which is a polynomial in $(t, x)$ with bounded coefficients depending on $(\bs v, \bs a)$, multiplied by $\partial_x^n H_j$ for some $n\in\bN$.
Similarly,
\begin{equation}
\sum_{|\bs p| + |\bs q| \leq l}\sum_{r + u \leq s}|\partial_{\bs v}^{\bs p}\partial_{\bs a}^{\bs q}\partial_t^r\partial_x^u (H_j(\bs v, \bs a; t, x) - \omega_{n_j})| \leq \eee^{-\delta(v_k t + a_k - x)_+}.
\end{equation}
Taking logarithms, we see that $\eee^{-\delta(x - v_j t - a_j)_+}\eee^{-\delta(v_k t + a_k - x)_+}\lesssim\eee^{-\delta_0 t}$,
if $T_0$ is large enough and $\delta_0$ small enough.

The proofs of \eqref{eq:cross21-step1} and \eqref{eq:cross22-step1} are very similar, so we skip them.

\noindent\textbf{Step 2.}
We proceed by induction with respect to $K$. If $K = 1$,
the left hand side is identically $0$, so let $K \geq 2$.

Note the formula
\begin{equation}
\begin{aligned}
&F(\phi_1 + \phi_2 - \phi_3) - F(\phi_1) - F(\phi_2) + F(\phi_3) \\
&= (\phi_1-\phi_3)(\phi_2-\phi_3)\int_0^1\int_0^1 F''(\sigma_1\phi_1 + \sigma_2\phi_2 + (1-\sigma_1 - \sigma_2)\phi_3)\ud \sigma_1\ud \sigma_2,
\end{aligned}
\end{equation}
obtained by applying the Fundamental Theorem twice.
Applying it with $\phi_2 := H_K$, $\phi_3 := \omega_{n_{k-1}}$ and
\begin{equation}
\phi_1 := \wt H := F(\omega_{n_0}) + \sum_{k=1}^{K-1} \big(F(H_k) - F(\omega_{n_{k-1}})\big),
\end{equation}
we obtain
\begin{equation}
\begin{aligned}
&F(H) - F(\wt H) - F(H_K) + F(\omega_{n_{K-1}}) \\
&= (\wt H - \omega_{n_{K-1}})(H_K - \omega_{n_{K-1}})\int_0^1\int_0^1 F''(\sigma_1\wt H + \sigma_2 H_K + (1-\sigma_1 - \sigma_2)\omega_{n_{K-1}})\ud \sigma_1\ud \sigma_2.
\end{aligned}
\end{equation}
The last integral is bounded, and all its derivatives as well.
Observing that $\wt H - \omega_{n_{K-1}} = \sum_{j=1}^{K-1}(H_{j} - \omega_{n_j})$
and using $K-1$ times \eqref{eq:cross-step1}, we obtain \eqref{eq:cross}.

\noindent\textbf{Step 3.}
In order to prove \eqref{eq:cross-2}, since $\partial_x H_j$ and all its derivatives
in $\bs v, \bs a, t, x$ are bounded by a polynomial in $t$, it suffices to check that
\begin{equation}
\label{eq:cross-22}
\Big(F(\omega_{n_0}) + \sum_{k=1}^K\big(F(H_k) - F(\omega_{n_{k-1}})\big) - F(H_j)\Big)\partial_x H_j\in B(C^l(A, \cH_{T_0, \delta_0}^s)), \qquad\text{for all }1 \leq j \leq K.
\end{equation}
But we observe that
\begin{equation}
F(\omega_{n_0}) + \sum_{k=1}^K\big(F(H_k) - F(\omega_{n_{k-1}})\big) - F(H_j) = \sum_{k=1}^{j-1}\big(F(H_k) - F(\omega_{n_{k}})\big) + \sum_{k=j+1}^{K}\big(F(H_k) - F(\omega_{n_{k-1}})\big),
\end{equation}
hence it suffices to prove that
\begin{equation}
\label{eq:cross-23}
\big(F(H_k) - F(\omega_{n_{k}})\big)\partial_x H_j\in B(C^l(A, \cH_{T_0, \delta_0}^s)), \qquad\text{for all }1 \leq k < j \leq K.
\end{equation}
and
\begin{equation}
\label{eq:cross-24}
\big(F(H_k) - F(\omega_{n_{k-1}})\big)\partial_x H_j\in B(C^l(A, \cH_{T_0, \delta_0}^s)), \qquad\text{for all }1 \leq j < k \leq K.
\end{equation}
We have $F(H_k) - F(\omega_{n_k}) = (H_k - \omega_{n_k})\int_0^1 F'((1-\sigma)\omega_{n_k} + \sigma H_k)\ud\sigma$, and the integral is bounded together with all its derivatives by a polynomial in $t$,
hence \eqref{eq:cross-23} follows from \eqref{eq:cross21-step1}.
Similarly, \eqref{eq:cross-24} follows from \eqref{eq:cross22-step1}.

\noindent\textbf{Step 4.} Finally, in order to prove \eqref{eq:cross-3},
we write
\begin{equation}
\label{eq:fund-thm-2}
F(H + g) - F(H) = g\int_0^1 F'(H + \sigma g)\ud\sigma.
\end{equation}
In order to prove continuity, we observe that $F'(H + \sigma g)$ is continuous from $A$ to the space $X$ defined by \eqref{eq:X-def}, uniformly in $\sigma$, so it suffices to apply \eqref{eq:bh-E}.

Next, we differentiate under the integral sign in \eqref{eq:fund-thm-2}, using again the fact that we can absorb
polynomials in $t$ by diminishing~$\delta$.
\end{proof}

\begin{lemma}
\label{lem:N-contr}
For any $A \ssubset S^{(K)}\times \bR^K$, $\delta_0$ small enough (depending on $A$)
and $0 < \delta < \delta_0$,
there exist $T_0$ and $\delta_2 > \delta_0$ such that if $T \geq T_0$, then
for all $g, \sh g \in B(\cH_{T, \delta_0})$ and $(\bs v, \bs a), (\sh{\bs v}, \sh{\bs a}) \in A$
\begin{gather}
\|N(\bs v, \bs a; \sh g) - N(\bs v, \bs a; g)\|_{\cH_{T, \delta_2}} \leq \|\sh g - g\|_{\cH_{T, \delta_0}}, \label{eq:N-contr-1} \\
\|N(\bs v, \sh{\bs a}; g) - N(\bs v, \bs a; g)\|_{\cH_T, \delta_0} \leq |\sh{\bs a} - \bs a|, \label{eq:N-contr-2} \\
\|N(\sh{\bs v}, \bs a; g) - N(\bs v, \bs a; g)\|_{\cH_T, \delta} \leq |\sh{\bs v} - \bs v|. \label{eq:N-contr-3}
\end{gather}
\end{lemma}
\begin{proof}
We have
\begin{equation}
\partial_g N(\bs v, \bs a; g) = -W''(H(\bs v, \bs a) + g) + V(\bs v, \bs a),
\end{equation}
thus
\begin{equation}
\begin{aligned}
& N(\bs v, \bs a; \sh g) - N(\bs v, \bs a; g) = \\
&= (\sh g - g)\int_0^1\big({-}W''(H(\bs v, \bs a) + (1-s)g + s\sh g) + V(\bs v, \bs a)\big)\ud s.
\end{aligned}
\end{equation}
We rewrite the integrand as
\begin{equation}
\big({-}W''(H(\bs v, \bs a) + (1-s)g + s\sh g) + W''(H(\bs v, \bs a))\big) + 
\big({-} W''(H(\bs v, \bs a)) + V(\bs v, \bs a)\big).
\end{equation}
By Lemma~\ref{lem:baby-cross}, the first term is uniformly bounded in $\cH_{T, \delta_0}$.
By Lemma~\ref{lem:cross} applied with $F = W''$,
the second term is uniformly bounded in $\cH_{T, \delta_3}$ for some $\delta_3 > 0$.
Hence, \eqref{eq:N-contr-1} follows from \eqref{eq:gh-E}.

We have
\begin{equation}
\partial_{a_k} N(\bs v, \bs a; g) = ({-}W''(H(\bs v, \bs a)+ g) + W''(H_k(\bs v, \bs a)) + W'''(H_k(\bs v, \bs a))g)\partial_{a_k}H_k(\bs v, \bs a).
\end{equation}
Similarly as above, Lemmas \ref{lem:baby-cross} and \ref{lem:cross} yield
uniform boundedness of the right hand side in $\cH_{T, \delta_0}$, proving \eqref{eq:N-contr-2}.
Analogously, $\partial_{v_k} N(\bs v, \bs a; g)$ is uniformly bounded in $\cH_{T, \delta}$,
which proves \eqref{eq:N-contr-3}.
\end{proof}
\begin{proposition}
\label{prop:construct}
For any $A \ssubset S^{(K)} \times \bR^K$, there exist $T_0, \delta_0 > 0$
such that the following is true.
For all $(\bs v, \bs a) \in A$ there exists $\Psi(\bs v, \bs a) \in B(\cH_{T_0, \delta_0})$
such that $\phi = H(\bs v, \bs a) + \Psi(\bs v, \bs a)$ is a solution of \eqref{eq:csf}.
\end{proposition}
\begin{proof}
Fix $(\bs v, \bs a) \in S^{(K)}\times \bR^K$, and let $R(\bs v, \bs a)$ be given by Lemma~\ref{lem:nonhom}.

Given $T \in \bR$, $\delta > 0$ and $g \in B(\cH_{T, \delta})$, we define
\begin{equation}
\Phi(g) := R(\bs v, \bs a)N(\bs v, \bs a; g).
\end{equation}
It follows that $\phi$ solves \eqref{eq:csf} if and only if $g = \Phi(g)$.

By Lemmas~\ref{lem:nonhom} and \ref{lem:N-contr}, if $T$ is large enough, then $\Phi$ is a contraction in $B(\cH_{T, \delta})$.
Since the constants in Lemmas~\ref{lem:nonhom} and \ref{lem:N-contr} are uniform in $(\bs v, \bs a) \in A$,
$T_0$ and $\delta_0$ can be chosen uniformly for $(\bs v, \bs a) \in A$.
\end{proof}

\begin{proof}[Proof of Theorem~\ref{thm:construct}]
The ``existence'' part follows from Proposition~\ref{prop:construct}, applied for $A$
being the singleton $\{(\bs v, \bs a)\}$.

Let $\psi \in \cH_{T, \delta}$ be such that $\phi = H(\bs v, \bs a) + \psi$ solves \eqref{eq:csf}.
Without loss of generality, we can assume $T \geq T_0$ and $\delta \leq \delta_0$,
so that $\Psi(\bs v, \bs a) \in B(\cH_{T, \delta})$.
Upon modifying $\delta$ and $T$, we can assume $\psi \in B(\cH_{T, \delta_0})$.

If $T$ is large enough, then $\Phi$ is a contraction on this set, implying uniqueness.
\end{proof}
\begin{lemma}
\label{lem:rhs-smooth}
For any $A \ssubset S^{(K)} \times \bR^K$, $l, s \in \bN$,
$T_0$ sufficiently large, $\delta_0$ sufficiently small (both depending on $A$) and $g \in C^k(A, \cH_{T_0, \delta_0}$)
\begin{equation}
-(W''(H + g) - W''(H_k))\partial_{a_k} H_k \in C^k(A, \cH_{T_0, \delta}^s)
\end{equation}
and
\begin{equation}
-(W''(H + g) - W''(H_k))\partial_{v_k} H_k \in C^k(A, \cH_{T_0, \delta}^s)
\end{equation}
for all $0 < \delta < \delta_0$.
\end{lemma}
\begin{proof}
By Lemma~\ref{lem:cross},
\begin{align}
-(W''(H) - W''(H_k))\partial_{a_k} H_k &\in C^k(A, \cH_{T_0, \delta_0}^s), \\
-(W''(H) - W''(H_k))\partial_{v_k} H_k &\in C^k(A, \cH_{T_0, \delta_0}^s).
\end{align}
By \eqref{eq:cross-3}, we have
\begin{align}
-(W''(H+g) - W''(H))\partial_{a_k} H_k &\in C^k(A, \cH_{T_0, \delta}^s), \\
-(W''(H+g) - W''(H))\partial_{v_k} H_k &\in C^k(A, \cH_{T_0, \delta}^s).
\end{align}
Taking the sum, we obtain the conclusion.
\end{proof}
\begin{lemma}
\label{lem:smooth}
Let $\Psi: A \times (T_0, \infty) \times \bR \to \bR$
be the function constructed in Proposition~\ref{prop:construct}.
Then $\partial_{a_k} \Psi$ exists for all $(\bs v, \bs a) \in A$ and solves the equation
\begin{equation}
\label{eq:dag}
\partial_t^2 \partial_{a_k} \Psi - \partial_x^2 \partial_{a_k} \Psi+ W''(H + \Psi)\partial_{a_k} \Psi
= -(W''(H + \Psi) - W''(H_k))\partial_{a_k} H_k.
\end{equation}
Moreover, $\partial_x \Psi = {-}\sum_{k=1}^K \partial_{a_k}\Psi$ and $\partial_t \Psi = {-}\sum_{k=1}^K v_k\partial_{a_k}\Psi$.

Futhermore, $\partial_{v_k} \Psi$ exists for all $(\bs v, \bs a) \in A$ and solves the equation
\begin{equation}
\label{eq:dvg}
\partial_t^2 \partial_{v_k}\Psi - \partial_x^2 \partial_{v_k} \Psi + W''(H + \Psi)\partial_{v_k} \Psi
= -(W''(H + \Psi) - W''(H_k))\partial_{v_k} H_k.
\end{equation}

Finally, $\partial_{a_k}\Psi, \partial_{v_k} \Psi \in C(A, \cH_{T_0, \delta})$ if $\delta > 0$ is small enough.
\end{lemma}
\begin{proof}
We first prove that $\bs v$, $\Psi(\bs v, \bs a)$ is continuous with respect to $(\bs v, \bs a) \in A$
as a $\cH_{T_0, \delta}$-valued map for any $0 < \delta < \delta_0$.
Let $(\bs v, \bs a), (\sh{\bs v}, \sh{\bs a}) \in A$. We have
\begin{equation}
\begin{aligned}
&\|N(\bs v, \sh{\bs a}; \Psi(\sh{\bs v}, \sh{\bs a})) - N(\bs v, \bs a; \Psi(\bs v, \bs a))\|_{\cH_{T_0, \delta}} \\
&\leq \|N(\sh{\bs v}, \sh{\bs a}; \Psi(\sh{\bs v}, \sh{\bs a})) - N(\bs v, \sh{\bs a}; \Psi(\sh{\bs v}, \sh{\bs a}))\|_{\cH_{T_0, \delta}}\\
&+\|N(\bs v, \sh{\bs a}; \Psi(\sh{\bs v}, \sh{\bs a})) - N(\bs v, \bs a; \Psi(\sh{\bs v}, \sh{\bs a}))\|_{\cH_{T_0, \delta}} \\
&+ \|N(\bs v, \bs a; \Psi(\sh{\bs v}, \sh{\bs a})) - N(\bs v, \bs a; \Psi(\bs v, \bs a))\|_{\cH_{T_0, \delta}}.
\end{aligned}
\end{equation}
By \eqref{eq:N-contr-3}, the first term converges to $0$ as $(\sh{\bs v}, \sh{\bs a}) \to (\bs v, \bs a)$.
By \eqref{eq:N-contr-2}, the second term as well.
The third term is estimated using \eqref{eq:N-contr-1}, which yields the bound $c_0\|\Psi(\sh{\bs v}, \sh{\bs a}))
-\Psi(\bs v, \bs a))\|_{\cH_{T_0, \delta}}$ where $c_0$ can be made arbitrarily small.
By \eqref{eq:bh-E}, we also have
\begin{equation}
\lim_{(\sh{\bs v}, \sh{\bs a}) \to (\bs v, \bs a)}\big\|(V(\sh{\bs v}, \sh{\bs a})- V(\bs v, \bs a))\Psi(\sh{\bs v}, \sh{\bs a})\big\|_{\cH_{T_0, \delta}} = 0,
\end{equation}
so the right hand side of
\begin{equation}
\begin{aligned}
&(\partial_t^2 - \partial_x^2 + V(\bs v, \bs a))(\Psi(\sh{\bs v}, \sh{\bs a}) - \Psi(\bs v, \bs a)) \\
&\quad=N(\sh{\bs v}, \sh{\bs a}; \Psi(\sh{\bs v}, \sh{\bs a})) - N(\bs v, \bs a; \Psi(\bs v, \bs a)) - (V(\sh{\bs v}, \sh{\bs a})- V(\bs v, \bs a))\Psi(\sh{\bs v}, \sh{\bs a})
\end{aligned}
\end{equation}
is bounded in $\cH_{T_0, \delta}$ by $o(1) + c_0\|\sh\Psi - \Psi\|_{\cH_{T_0, \delta}}$, and Lemma~\ref{lem:nonhom} allows to absorb the second term, thus finishing the proof of $\lim_{(\sh{\bs v}, \sh{\bs a}) \to (\bs v, \bs a)}\|\Psi(\sh{\bs v},\sh{\bs a}) - \Psi(\bs v, \bs a)\|_{\cH_{T_0, \delta}} = 0$.

Let $k \in \{1, \ldots, K\}$ and define $\Theta: A \times (T_0, \infty) \times \bR \to \bR$
as the solution of
\begin{equation}
\label{eq:theta}
\partial_t^2 \Theta - \partial_x^2 \Theta + W''(H + \Psi)\Theta
= -(W''(H + \Psi) - W''(H_k))\partial_{a_k} H_k,
\end{equation}
given by Lemma~\ref{lem:reg-lin}.
Set $b := W''(H + \Psi) - V$. Since $\Psi \in C^0(A, \cH_{T_0, \delta})$, \eqref{eq:cross-3} yields
$W''(H + \Psi) - W''(H) \in C^0(A, \cH_{T_0, \delta})$ (upon diminishing $\delta$ if necessary).
By \eqref{eq:cross}, we also have $W''(H) - V \in C^0(A, \cH_{T_0, \delta})$,
thus $b \in B(C^0(A, \cH_{T_0, \delta}))$ (again modifying $\delta$ and $T_0$ if necessary).
By Lemma~\ref{lem:rhs-smooth} and Lemma~\ref{lem:reg-lin}, $\Theta \in C^0(A, \cH_{T_0, \delta})$.

Let $(\bs v, \bs a) \in A$.
We will verify that if $T_0$ is large enough (depending on $A$),
then $\partial_{a_k}\Psi(\bs v, \bs a; t, x)$ exists for all $(t, x) \in (T_0, \infty) \times \bR$
and $\partial_{a_k}\Psi(\bs v, \bs a) = \Theta(\bs v, \bs a)$.

To this end, for $0 < |\epsilon|$ small enough, consider $\Xi_\epsilon : (T_0, \infty) \times \bR \to \bR$
given by
\begin{equation}
\Xi_\epsilon(t, x) := \Psi(\bs v, \bs a + \epsilon \bs e_k, t, x) - \Psi(\bs v, \bs a, t, x) -
\epsilon\,\Theta(\bs v, \bs a, t, x),
\end{equation}
where $\bs e_k := (0, \ldots, 1, \ldots, 0)$ is the $k$-th element of the standard basis of $\bR^K$.
Our goal is to verify that
\begin{equation}
\label{eq:Xi-conv}
\lim_{\epsilon \to 0}|\epsilon|^{-1}\|\Xi_\epsilon\|_{\cH_{T_0, \delta}} = 0.
\end{equation}
In the computation below, we denote $\sh{\bs a} := \bs a + \epsilon \bs e_k$ and $\sh \Psi(\bs v, \bs a, t, x) := \Psi(\bs v, \sh{\bs a}, t, x)$. We consider $(\bs v, \bs a)$ as being fixed, but below we omit the argument $(\bs v, \bs a)$ for the sake of brevity, thus we write $H$ instead of $H(\bs v, \bs a)$ etc.
We observe that $\Xi_\epsilon$ solves the equation
\begin{equation}
\begin{aligned}
&(\partial_t^2 - \partial_x^2 + V)\Xi_\epsilon = (\partial_t^2 - \partial_x^2 + V)(\sh \Psi - \Psi - \epsilon \Theta)  \\
&= (\partial_t^2 - \partial_x^2)\sh\Psi - (\partial_t^2 - \partial_x^2)\Psi
 + V(\sh\Psi - \Psi) 
-\epsilon(\partial_t^2 - \partial_x^2 - V)\Theta \\
&= {-}W'(\sh H + \sh\Psi) +W'(H + \Psi) + \sum_{j=1}^K (W'(\sh H_j) - W'(H_j))+ V(\sh\Psi - \Psi) \\
&\qquad+ \epsilon(W''(H+\Psi) - V)\Theta + \epsilon(W''(H+\Psi) - W''(H_k))\partial_{a_k}H_k,
\end{aligned}
\end{equation}
where in the last equality we use \eqref{eq:theta}.

We claim that the right hand side is bounded in $\cH_{T_0, \delta}$ by $o(|\epsilon|) + c_0\|\Xi_\epsilon\|_{\cH_{T_0, \delta}}$, with $c_0$ small (upon adjusting $T_0$ and $\delta$).
Lemma~\ref{lem:nonhom} will then allow to conclude.

We further rearrange the terms of the right hand side above as follows, using the fact that $\sh H_j = H_j$ for $j \neq k$:
\begin{equation}
\label{eq:last}
\begin{aligned}
&-\big(W'(H+\sh\Psi) - W'(H+\Psi) - W''(H+\Psi)(\sh\Psi - \Psi)\big) \\
&- \big(W'(\sh H + \sh \Psi) - W'(H+\sh\Psi)\big) + \big(W'(\sh H_k) - W'(H_k)\big) \\
&+ \epsilon(W''(H+\sh\Psi) - W''(H_k))\partial_{a_k}H_k \\
&- \epsilon(W''(H + \sh\Psi) - W''(H + \Psi))\partial_{a_k}H_k \\
&- \big(W''(H+\Psi) - V\big)(\sh \Psi - \Psi - \epsilon\Theta).
\end{aligned}
\end{equation}
Consider the last line, noting that $\sh \Psi - \Psi - \epsilon\Theta = \Xi_\epsilon$. 
As in the proof of Lemma~\ref{lem:N-contr}, we write
\begin{equation}
{-}W''(H+\Psi) + V = \big({-}W''(H+\Psi) + W''(H)\big) +
\big({-}W''(H) + V\big).
\end{equation}
Using Lemmas~\ref{lem:baby-cross} and \ref{lem:cross},
we obtain a bound in $\cH_{T_0, \delta_3}$ for some $\delta_3 > 0$.
Enlarging $T_0$ if necessary and using \eqref{eq:gh-E}, we conclude that
\begin{equation}
\|\big({-}W''(H+\Psi) + V\big)\Xi_\epsilon\|_{\cH_{T_0, \delta}}
\leq c_0\|\Xi_\epsilon\|_{\cH_{T_0, \delta}}.
\end{equation}
Regarding the fourth line, using \eqref{eq:taylor-1} together with the fact that $\lim_{\epsilon \to 0}\|\sh \Psi - \Psi\|_{\cH_{T_0, \delta}} = 0$, we obtain the bound $o(\epsilon)$ in $\cH_{T_0, \delta}$.

Consider now the first line of \eqref{eq:last}. By \eqref{eq:taylor-2}, we have
\begin{equation}
\begin{aligned}
&\|W'(H+\sh\Psi) - W'(H+\Psi) - W''(H+\Psi)(\sh\Psi - \Psi)\|_{\cH_{T_0, \delta}} \lesssim \|\sh\Psi - \Psi\|_{\cH_{T_0, \delta}}^2 \\
&\qquad\ll \|\sh\Psi - \Psi\|_{\cH_{T_0, \delta}} = \|\epsilon\Theta + \Xi_\epsilon\|_{\cH_{T_0, \delta}} \lesssim \epsilon + \|\Xi_\epsilon\|_{\cH_{T_0, \delta}}.
\end{aligned}
\end{equation}

Using the Fundamental Theorem and $\sh H - H = \sh H_k - H_k$, we rewrite the second line of \eqref{eq:last} as
\begin{equation}
\label{eq:last-2}
\begin{aligned}
&(\sh H_k - H_k)\int_0^1\big({-}W''(H + \sh\Psi + \sigma(\sh H_k - H_k)) + W''(H_k + \sigma(\sh H_k - H_k))\big)\ud\sigma \\
&\quad = {-}(\sh H_k - H_k)(H - H_k + \sh\Psi)\int_0^1\int_0^1 W'''(H_k + \tau(H - H_k + \sh\Psi) + \sigma(\sh H_k - H_k))\ud \tau\ud\sigma.
\end{aligned}
\end{equation}
Since
\begin{equation}
\|H_k(\bs v, \sh{\bs a}) - H_k(\bs v, \bs a)\|_X = \Big\|\epsilon\int_0^1 \partial_{a_k}H_k(\bs v, \bs a + \sigma\epsilon\bs e_k)\ud \sigma\Big\|_X \lesssim \epsilon,
\end{equation}
\eqref{eq:bh-E} yields $\|(\sh H - H)\sh\Psi\|_{\cH_{T_0, \delta}} \lesssim \epsilon$.
Using the exponential decay of $\partial_{a_k}H_k$ and its derivatives for $|x - a_k - v_k t| \gg 1$, we also obtain
$\|(\sh H_k - H_k)(H - H_k)\|_{\cH_{T_0, \delta}} \lesssim \epsilon$.
Moreover, since $\|\sh H_k - H_k\|_X \lesssim \epsilon$, another application of the Fundamental Theorem gives
\begin{equation}
\|W'''(H_k + \tau(H - H_k + \sh\Psi) + \sigma(\sh H_k - H_k)) - W'''(H_k + \tau(H - H_k + \sh\Psi))\|_X \lesssim \epsilon.
\end{equation}
Thus, up to terms bounded in  $\cH_{T_0, \delta}$ by $O(\epsilon^2)$, \eqref{eq:last-2} is the same as
\begin{equation}
{-}(\sh H_k - H_k)(H - H_k + \sh\Psi)\int_0^1 W'''(H_k + \tau(H - H_k + \sh\Psi))\ud \tau.
\end{equation}
Adding the third line of \eqref{eq:last}, we obtain
\begin{equation}
{-}(\sh H_k - H_k - \epsilon \partial_{a_k}H_k)(H - H_k + \sh\Psi)\int_0^1 W'''(H_k + \tau(H - H_k + \sh\Psi))\ud \tau.
\end{equation}
Since
\begin{equation}
\|\sh H_k - H_k - \epsilon\partial_{a_k}H_k\|_X = \Big\|\epsilon^2\int_0^1 (1-\sigma)\partial_{a_k}^2 H_k(\bs v, \bs a + \sigma\epsilon\bs e_k)\ud \sigma\Big\|_X \lesssim \epsilon^2,
\end{equation}
\eqref{eq:bh-E} yields $\|(\sh H_k - H_k-\epsilon\partial_{a_k}H_k)\sh\Psi\|_{\cH_{T_0, \delta}} \lesssim \epsilon^2$.
Using the exponential decay of $\partial_{a_k}^2H_k$ and its derivatives for $|x - a_k - v_k t| \gg 1$, we also obtain
$\|(\sh H_k - H_k - \epsilon\partial_{a_k}H_k)(H - H_k)\|_{\cH_{T_0, \delta}} \lesssim \epsilon^2$.

This finished the proof of the required estimate of \eqref{eq:last}.

The proofs of existence of the other partial derivatives are similar, so we skip them.
The formulas for $\partial_x \Psi$ and $\partial_t \Psi$ follow from the fact that
$\partial_x H = {-}\sum_{k=1}^K \partial_{a_k}H$ and $\partial_t H = {-}\sum_{k=1}^K v_k\partial_{a_k}H$.
\end{proof}
\begin{proof}[Proof of Theorem~\ref{thm:smooth}]
Induction with respect to $s$ and $l$.
The induction step follows from Lemma~\ref{lem:reg-lin} and Lemma~\ref{lem:smooth}.
\end{proof}

\begin{proof}[Proof of Theorem~\ref{thm:lorentz}]
\label{proof:lorentz}
Fix $(\bs v, \bs a) \in S^{(K)} \times \bR^K$.
We should prove that, applying a space-time translation or a Lorentz transform to $H(\bs v, \bs a) + \Psi(\bs v, \bs a)$
yields a function converging exponentially in time to $H(\bs v', \bs a')$, where $(\bs v', \bs a')$
is given by \eqref{eq:prime-formula}.
The uniqueness part of Theorem~\ref{thm:construct} will then yield the conclusion.

In the case of space-time translations, the claim is clear, so let us consider a Lorentz boost
with velocity $v$, $(t, x) = (\gamma (t' + vx'), \gamma(x' + vt'))$, $(t', x') = (\gamma(t - vx), \gamma(x-vt))$.
Without loss of generality, we can assume $v \in (0, 1)$.
Let $g := \Psi(\bs v, \bs a)$, $\phi := H(\bs v, \bs a) + g$, $\wt \phi(t', x') := \phi(t, x)$, $h := \wt\phi - H(\bs v', \bs a')$.
Recall that $g$ is defined for all $(t, x) \in \bR^2$ thanks to the global well-posedness.
Since $H(\bs v', \bs a'; t', x') = H(\bs v, \bs a; t, x)$, we have $h(t', x') = g(t, x)$. We wish to prove that
\begin{equation}
\int_\bR \big((\partial_{t'}h(t', x'))^2 + (\partial_{x'}h(t', x'))^2 + h(t', x')^2\big)\ud x' \lesssim \eee^{-\delta t'}
\end{equation}
for some $\delta > 0$. Equivalently, observing that $\gamma(t' + vx') = \gamma^{-1}t'+ vx$, where $\gamma := (1-v^2)^{-1/2}$,
\begin{equation}
\label{eq:lorentz-exp}
\int_\bR \big((\partial_{t}g(\gamma^{-1}t' + vx, x))^2 + (\partial_{x}g(\gamma^{-1}t' + vx, x))^2 + g(\gamma^{-1}t' + vx, x)^2\big)\ud x \lesssim \eee^{-\delta t'}.
\end{equation}

We consider separately the regions $x \geq -\frac{t'}{2\gamma v}$ and $x \leq -\frac{t'}{2 \gamma v}$.
In the first region, it suffices to use the exponential in time decay of $g$ and its derivatives,
uniformly in $x$, which follows from Theorem~\ref{thm:smooth} and the embedding $H^1(\bR) \subset L^\infty(\bR)$.

Consider the second region. Set $(t_0, x_0) := (t'/(2\gamma), -t'/(2v\gamma))$
and let $\Delta$ be the cone in the $(t, x)$ plane $\bR\times \bR$ with vertex at $(t_0, x_0)$,
delimited by the half-lines $\{(t_0, x): x \leq x_0\}$ and $\{(t_0 + v(x-x_0), x): x \leq x_0\}$.
An elementary computation shows that $(t, x) \in \Delta$ implies $x - v_k t - a_k \lesssim x$,
thus
\begin{equation}
|H(\bs v, \bs a; t, x) - \omega_{n_0}| + |\partial_t H(\bs v, \bs a; t, x)| + |\partial_x H(\bs v, \bs a; t, x)| \lesssim \eee^{-\delta x},
\end{equation}
uniformly for $(t, x) \in \Delta$, with a constant independent of $t_0$.
In particular, proving that
\begin{equation}
\label{eq:lorentz-exp-2}
\int_{-\infty}^{x_0} \big((\partial_{t}g(\gamma^{-1}t' + vx, x))^2 + (\partial_{x}g(\gamma^{-1}t' + vx, x))^2 + g(\gamma^{-1}t' + vx, x)^2\big)\ud x \lesssim \eee^{-\delta t'}.
\end{equation}
is equivalent, since $W(\omega_{n_0} + g) \simeq g^2$ for $|g|$ small, to verifying that
\begin{equation}
\label{eq:lorentz-exp-3}
\int_{-\infty}^{x_0} \Big(\frac 12 (\partial_{t}\phi(\gamma^{-1}t' + vx, x))^2 + \frac 12(\partial_{x}\phi(\gamma^{-1}t' + vx, x))^2 + W(\phi(\gamma^{-1}t' + vx, x))\Big)\ud x \lesssim \eee^{-\delta t'}.
\end{equation}
For the same reason, the exponential decay in time of $g$ implies
\begin{equation}
\label{eq:lorentz-exp-4}
\int_{-\infty}^{x_0} \Big(\frac 12(\partial_{t}\phi(t_0, x))^2 + \frac 12(\partial_{x}\phi(t_0, x))^2 + W(\phi(t_0, x))\Big)\ud x \lesssim \eee^{-\delta t'}.
\end{equation}
By the Green's theorem applied for the smooth vector field
$(\frac 12 (\partial_t \phi)^2 + \frac 12 (\partial_x \phi)^2 + W(\phi), -\partial_t \phi\partial_x \phi)$,
in the region $\Delta$, \eqref{eq:lorentz-exp-4} implies \eqref{eq:lorentz-exp-3}.
\end{proof}

\providecommand{\noopsort}[1]{}

\end{document}